\documentclass[11pt]{article}
\usepackage[]{amsmath,amssymb,amsfonts,latexsym,amsthm,enumerate,fullpage}
\usepackage[latin9]{inputenc}
\usepackage{amsmath}
\usepackage{amssymb}
\usepackage{breqn}
\usepackage{url}
\usepackage{graphicx}
\usepackage[pdfstartview=FitH]{hyperref}
\usepackage{amsthm}
\usepackage{hyperref}
\usepackage{url}
\usepackage{enumitem}
\usepackage{bbm}
\usepackage{verbatim}

\newcommand{\I}{\operatorname{I}}
\newcommand{\type}{\operatorname{type}}

\newcommand{\Stab}{\operatorname{Stab}}
\newcommand{\Sys}{\operatorname{Sys}}
\newcommand{\Fill}{\operatorname{Fill}}
\newcommand{\SFill}{\operatorname{SFill}}
\newcommand{\Exp}{\operatorname{Exp}}
\newcommand{\supp}{\operatorname{supp}}
\newcommand{\Cone}{\operatorname{Cone}}
\newcommand{\Crad}{\operatorname{Crad}}
\newcommand{\Area}{\operatorname{Area}}
\newcommand{\Dehn}{\operatorname{Dehn}}
\newcommand{\Vol}{\operatorname{Vol}}
\newcommand{\dist}{\operatorname{dist}}

\newcommand{\Free}{\operatorname{Free}}
\newcommand{\Unip}{\operatorname{Unip}}

\makeatletter

\begin{document}
\newtheorem{theorem}{Theorem}[section]
\newtheorem{lemma}[theorem]{Lemma}
\newtheorem{definition}[theorem]{Definition}
\newtheorem{claim}[theorem]{Claim}
\newtheorem{example}[theorem]{Example}
\newtheorem{conjecture}[theorem]{Conjecture}
\newtheorem{remark}[theorem]{Remark}
\newtheorem{proposition}[theorem]{Proposition}
\newtheorem{corollary}[theorem]{Corollary}
\newtheorem{observation}[theorem]{Observation}
\newcommand{\subscript}[2]{$#1 _ #2$}

\title{Coboundary and cosystolic expansion from strong symmetry}

\author{
Tali Kaufman
\footnote{Department of Computer Science, Bar-Ilan University, Ramat-Gan, 5290002, Israel, email:kaufmant@mit.edu, research supported by ERC and BSF.}
\and
Izhar Oppenheim
\footnote{Department of Mathematics, Ben-Gurion University of the Negev, Be'er Sheva 84105, Israel, email: izharo@bgu.ac.il }
}

\maketitle

\begin{abstract}
Coboundary and cosystolic expansion are notions of expansion that generalize the Cheeger
constant or edge expansion of a graph to higher dimensions. The classical Cheeger inequality
implies that for graphs edge expansion is equivalent to spectral expansion.  In higher dimensions this is not the case: a simplicial complex can be spectrally expanding but not have high dimensional edge-expansion.
The phenomenon of high dimensional edge expansion in higher dimensions is much more
involved than spectral expansion, and is far from being understood. In particular, prior to this work, the only known bounded degree cosystolic expanders known were derived from the theory of buildings that is far from being elementary.

In this work we study high dimensional complexes which are {\em strongly symmetric}. Namely, there is a group that acts transitively on top dimensional cells of the simplicial complex [e.g., for graphs it corresponds to a group that acts transitively on the edges]. Using the strong symmetry, we develop a new machinery to prove coboundary and cosystolic expansion.

It was an open question whether the recent elementary construction of bounded degree spectral high dimensional expanders based on coset complexes give rise to bounded degree cosystolic expanders. In this work we answer this question affirmatively. We show that these complexes give rise to bounded degree cosystolic expanders in dimension two, and that their links are (two-dimensional) coboundary expanders. We do so by exploiting the strong symmetry properties of the links of these complexes using a new machinery developed in this work.

Previous works have shown a way to bound the co-boundary expansion using strong symmetry in the special situation of "building like" complexes. Our new machinery shows how to get co-boundary expansion for \emph{general} strongly symmetric coset complexes, which are not necessarily "building like", via studying the (Dehn function of the) presentation of the symmetry group of these complexes.

\end{abstract}

\section{Introduction}

High dimensional expansion is a vibrant emerging field that has found
applications to PCPs \cite{Dinur2017Highdimensionalexpanders} and
property testing \cite{Kaufman2014HDTesting}, to counting problems and matroids \cite{Anari2019matroids}, to list decoding \cite{Dinur2019listdecoding}, and recently to a breakthrough construction of decodable quantum error correcting codes that outperform the state-of-the art previously known codes \cite{EvraKaufmanZemor}. We refer the reader
to \cite{lubotzky2017high} for a recent (but already outdated) survey.

The term high dimensional expander means a simplicial complex that have expansion properties that are analogous to expansion in a graph. Nevertheless, the question of what is a high dimensional expander is still unclear. There is a spectral definition of high dimensional expanders that generalizes the spectral definition of expander graphs and a geometrical/topological definition that generalizes the notion of edge expansion (or Cheeger constant) of a graph. For a graph the spectral and the geometric definitions of expansion are known to be equivalent (via the celebrated Cheeger inequality) while in high dimensions the spectral and geometric definitions are known to be NOT equivalent (see \cite[Theorem 4]{GunW} and \cite{SKM}).

The aim of this paper is to present elementary constructions of new families of $2$-dimensional simplicial complexes with high dimensional edge expansion, and in particular, of new elementary {\em bounded degree} families of cosystolic expanders (see exact definition below). The question of giving an elementary construction of a family of bounded degree spectrally high dimensional expanders got recently a satisfactory answer. Namely, it was understood that such a family needs to obey a specific local spectral criterion and in \cite{KOconstruction} we used this understanding in order to construct elementary families of high dimensional {\em spectrally} expanding families (prior non-elementary constructions were known). Here, we further study of examples of \cite{KOconstruction}, and show that they also give rise to bounded degree cosystolic expanders:
\begin{theorem}[New cosystolic expanders, Informal, see also Theorem \ref{new cosystolic expanders thm intro}]
\label{new cosystolic expanders thm intro informal}
For every large enough odd prime power $q$, the family of $2$-skeletons of the $3$-dimensional local spectral expanders constructed in \cite{KOconstruction} using elementary matrices over $\mathbb{F}_q [t]$ is a family of bounded degree cosystolic expanders.
\end{theorem}
Prior to this work, the known examples of bounded degree cosystolic expanders arose from the theory of Bruhat-Tits buildings and were far from being elementary.

Relying on the work of the first named author and Evra (see Theorem \ref{EK criterion thm} below), the proof of this Theorem boils down to proving that the links of our construction are coboundary expanders and that their coboundary expansion can be bounded independently of $q$ (i.e., that the coboundary does not deteriorate as $q$ increases). Thus, the real problem is bounding the coboundary expansion of the links. This goal is achieved utilizing the fact that the links are strongly symmetric coset complexes.

\paragraph*{Coboundary expansion for strongly symmetric (coset) complexes.} We call a simplicial complex is strongly symmetric if it has a symmetry group acting transitively on top dimensional simplices. As noted above, our problem is to show that the links in our examples are coboundary expanders. In the graph setting, there is a classical Theorem (see Theorem \ref{thm for symmetric graphs} below) stating that for a strongly symmetric graph the Cheeger constant can be bounded from below by $\frac{1}{2D}$, where $D$ denotes  the diameter of the graph.

We generalize this idea: we define a high dimensional notion of radius and show that for strongly symmetric complexes, this radius can be used to bound the coboundary expansion. We then show that this radius can be bounded using filling constants of the complex. These ideas of bounding the coboundary expansion for symmetric complexes using filling constants already appeared implicitly in Gromov's work \cite{Grom} and in the work of Lubotzky, Meshulam and Mozes \cite{LMM}. However, these previous works considered the setting of spherical buildings and ``building-like complexes'' and thus bounding the filling constants in these examples were relatively simple due to the existence of apartments in the building (or ``apartment-like'' sub-complexes in ``building-like'' complexes). In our setting, we consider a more general situation (not assuming ``apartment-like'' sub-complexes) and thus bounding the filling constants becomes a much harder task.

What helps to solve this harder problem of bounding the filling constants is working with a strongly symmetric \emph{coset complexes} (see Definition \ref{coset complex def}). We note that this is not a very restrictive assumption - under some mild assumptions, every strongly symmetric complex is a coset complex (see Proposition \ref{symmetric complexes and coset complexes prop}). For a coset complex one can fully reconstruct the complex via its symmetry group and its subgroup structure. Thus every geometrical/topological property of a coset complex (including coboundary expansion) is encoded in some way in the presentation of its symmetry group. Using this philosophy, we are able to prove a bound for filling constants for two dimensional coset complex in terms of the presentation of its symmetry group (namely, in terms on its Dehn function - see Definition \ref{Dehn func def}). Thus, for two dimensional coset complexes, we get a bound on the coboundary expansion in terms of presentation-theoretic properties of the symmetry group.

\paragraph*{Coboundary expansion of the links in our construction.} If follows from our work described above that in order to show that the links in our construction are coboundary expanders, we should verify a presentation-theoretic property for their symmetry group (namely, to bound its Dehn function). Luckily for us, the symmetry group of the links in our construction is a generalization of the group of unipotent groups over finite fields. For the finite field case, the presentation of these unipotent groups was studied by Biss and Dasgupta \cite{BissD}. Using their ideas, we are able to show that the symmetry groups of links in our construction fulfil the presentation-theoretic condition that allows us to bound their coboundary expansion. Namely, we prove the following:
\begin{theorem}[New coboundary expanders, Informal, see also Theorem \ref{new coboundary expanders thm intro}]
\label{new cosystolic expanders thm intro informal}
For every odd prime power $q$, the links of the $3$-dimensional local spectral expanders constructed in \cite{KOconstruction} using elementary matrices over $\mathbb{F}_q [t]$ are coboundary expanders and their coboundary expansion can be bounded from below independently of $q$.
\end{theorem}




\subsection{Simplicial complexes}

An $n$-dimensional simplicial complex $X$ is a hypergraph whose maximal hyperedges are of size $n+1$, and which is closed under containment. Namely, for every hyperedge $\tau$ (called a face) in $X$, and every $\eta\subset\tau$, it must be that $\eta$ is also in $X$. In particular, $\emptyset \in X$. For example, a graph is a $1$-dimensional simplicial complex. Let $X$ be a simplicial complex, we fix the following terminology/notation:
\begin{enumerate}
\item $X$ is called {\em pure $n$-dimensional} if every face in $X$ is contained in some face of size $n+1$.
\item The set of all $k$-faces (or $k$-simplices) of $X$ is denoted $X(k)$, and we will be using the convention in which $X(-1) = \{\emptyset\}$.
\item For $0 \leq k \leq n$, the $k$-skeleton of $X$ is the $k$-dimensional simplicial complex $X(0) \cup X(1) \cup ... \cup X(k)$. In particular, the $1$-skeleton of $X$ is the graph whose vertex set is $X(0)$ and whose edge set is $X(1)$.
\item For a simplex $\tau \in X$, the link of $\tau$, denoted $X_\tau$ is the complex
$$\lbrace \eta \in X : \tau \cup \eta \in X, \tau \cap \eta = \emptyset \rbrace.$$
We note that is $\tau \in X(k)$ and $X$ is pure $n$-dimensional, then $X_\tau$ is pure $(n-k-1)$-dimensional.
\item A family of of pure $n$-dimensional simplicial complexes  $\lbrace X^{(s)} \rbrace_{s \in \mathbb{N}}$ is said to have bounded degree if there is a constant $L>0$ such that for every $s \in \mathbb{N}$ and every vertex $v$ in $X^{(s)}$, $v$ is contained in at most $L$ $n$-dimensional simplices of $X^{(s)}$.
\end{enumerate}

\subsection{The coboundary/cosystolic expansion and high order Cheeger constants}

Let us recall the geometric notion of expansion in graphs known as the edge expansion or Cheeger constant of a graph:
\begin{definition}[Cheeger constant of a graph]
For a graph $X=(V,E):$
\[h(X) := \mbox{min}_{A \neq \emptyset,V}\frac{|E(A,\bar{A})|}{\mbox{min}\{w(A),w (\bar{A})\}},\]
where for a set of vertices $U \subsetneqq V$, $w(U)$ denotes is the sum of the degrees of the vertices in $U$.
\end{definition}

The generalization of the Cheeger constant to higher dimensions originated in the works of Linial, Meshulam and Wallach (\cite{LM}, \cite{MW}) and independently in the work of Gromov (\cite{Grom}) and is now known as {\em coboundary expansion}. Later, a weaker variant of high dimensional edge expansion known as {\em cosystolic expansion} arose in order to answer questions regarding topological overlapping.



In order to define coboundary and cosystolic expansion, we also need some terminology. Let $X$ be an $n$-dimensional simplicial complex. Fix the following notations/definitions:
\begin{enumerate}
\item The space of $k$-cochains denoted $C^k (X) = C^k(X,\mathbb{F}_2)$ is the $\mathbb{F}_2$-vector space of functions from $X(k)$ to $\mathbb{F}_2$.
\item The {\em coboundary map} $d_k: C^k(X,\mathbb{F}_2) \rightarrow C^{k+1}(X,\mathbb{F}_2)$ is defined as:
\begin{equation*}
d_k( \phi)(\sigma)=\sum_{\tau \subset \sigma, |\tau|=|\sigma|-1} \phi (\tau),
\end{equation*}
\item The spaces of $k$-coboundaries and $k$-cocycles are subspaces of $C^k (X)$ defined as:

$B^k (X) = B^k(X,\mathbb{F}_2) = \mbox{Image}(d_{k-1})=$\mbox{ the space of $k$-coboundaries}.

$Z^k (X) =Z^k(X,\mathbb{F}_2) = \mbox{Ker}(d_{k})=$\mbox{ the space of $k$-cocycles}.
\item The function $ w : \bigcup_{k=-1}^n X(k) \rightarrow \mathbb{R}_+$ is defined as
$$\forall \tau \in X(k), w(\tau) = \frac{\vert \lbrace \sigma \in X(n) : \tau \subseteq \sigma \rbrace \vert}{{n+1 \choose k+1} \vert X(n) \vert}.$$
We note that $\sum_{\tau \in X(k)} w(\tau) =1$.
\item For every $\phi \in C^k (X)$, $w(\phi)$ is defined as
$$w(\phi) = \sum_{\tau \in \supp (\phi)} w(\tau).$$
\item For every $0 \leq k \leq n-1$, define the following $k$-expansion constants:
$$\Exp^k_b (X) =  \min \left\lbrace \frac{w (d_k \phi)}{\min_{\psi \in B^k (X)} w(\phi + \psi)} : \phi \in C^k (X) \setminus B^k (X) \right\rbrace.$$
$$\Sys^k (X) = \min \left\lbrace w (\psi) : \psi \in Z^k (X) \setminus B^k (X) \right\rbrace,$$
and
$$\Exp^k_z (X) =  \min \left\lbrace \frac{w (d_k \phi)}{\min_{\psi \in Z^k (X)} w(\phi + \psi)} : \phi \in C^k (X) \setminus Z^k (X) \right\rbrace.$$
\end{enumerate}

After these notations, we can define coboundary/cosystolic expansion:
\begin{definition}[Coboundary expansion]
Let $\varepsilon >0$ be a constant. We say that $X$ is an $\varepsilon$-coboundary expander if for every $0 \leq k \leq n-1$, $\Exp^k_b (X) \geq \varepsilon$.
\end{definition}

\begin{remark}
We leave it for the reader to verify that in the case where $X$ is a graph, i.e., the case where $n=1$, $\Exp^0_b (X)$ is exactly the Cheeger constant of $X$. Thus, we think of $\Exp^k_b (X)$ as the $k$-dimensional Cheeger constant of $X$.
\end{remark}

\begin{definition}[Cosystolic expansion]
Let $\varepsilon >0, \mu >0$ be constants and $X$ an $n$-dimensional simplicial complex $X$. We say that $X$ is a $(\varepsilon , \mu)$-cosystolic expander if for every $0 \leq k \leq n-1$, $\Exp^k_z (X) \geq \varepsilon$ and $\Sys^k (X) \geq \mu$.
\end{definition}

\begin{remark}
We note that if $\Exp^k_b (X) >0$, then it can be shown that $B^k (X) = Z^k (X)$ and thus $\Exp^k_b (X) = \Exp^k_z (X)$. However, there are examples of simplicial complexes with $\Exp^k_b (X)  =0$ and $\Exp^k_z (X) >0, \Sys^k (X) >0$.
\end{remark}

As in expander graphs, we are mainly interested in a family of bounded degree cosystolic expanders (and not a single complex that is a cosystolic expander):

\begin{definition}[A family of bounded degree cosystolic expanders]
A family of $n$-dimensional simplicial complexes $\lbrace Y^{(s)} \rbrace_{s \in \mathbb{N}}$ is a family of \textit{bounded degree cosystolic expanders} if:
\begin{itemize}
\item The number of vertices of $Y^{(s)}$ tends to infinity with $s$.
\item  $\lbrace Y^{(s)} \rbrace_{s \in \mathbb{N}}$ has bounded degree.
\item There are universal constants $\varepsilon >0 , \mu >0$ such that for every $s$, $Y^{(s)}$ is a $(\varepsilon , \mu)$-cosystolic expander.
\end{itemize}
\end{definition}

\begin{remark}
The motivation behind the definition of a family of cosystolic expanders is to proved a family of bounded degree complexes that have the topological overlapping property (see Definition \ref{top exp def} below).
\end{remark}

\subsection{The Evra-Kaufman criterion for cosystolic expansion}

In \cite{EvraK}, Evra and the first named author gave a criterion for cosystolic expansion. In order to state this criterion, we will need the following definition:

\begin{definition}[Local spectral expansion]
For $\lambda \geq 0$, a pure $n$-dimensional simplicial complex $X$ is called a (one-sided) $\lambda$-local spectral expander if for $-1 \leq k \leq n-2$ and every $\tau \in X(k)$, the one-skeleton of $X_\tau$ is a connected graph and the second largest eigenvalue of the random walk on the one-skeleton of $X_\tau$ is less or equal to $\lambda$.
\end{definition}

The idea behind the Evra-Kaufman criterion for cosystolic expansion is the following: For we can deduce cosystolic expansion from local spectral expansion and local coboundary expansion (i.e., coboundary expansion in the links) given that the local spectral expansion is ``strong enough'' so it ``beats'' the local coboundary expansion. More formally:

\begin{theorem}\cite[Theorem 1]{EvraK} [Evra-Kaufman criterion for cosystolic expansion]
\label{EK criterion thm}
For every $\varepsilon ' >0$ and $ n \geq 3$ there are $\mu (n, \varepsilon ') > 0, \varepsilon (n, \varepsilon ') >0$ and $\lambda (n, \varepsilon ') > 0$ such that for every pure $n$-dimensional simplicial complex if
\begin{itemize}
\item $X$ is a $\lambda$-local spectral expander.
\item For every $0 \leq k \leq n-2$ and every $\tau \in X(k)$, $X_\tau$ is a $\varepsilon '$-coboundary expander.
\end{itemize}
Then the $(n-1)$-skeleton of $X$ is a $(\varepsilon , \mu)$-cosystolic expander.
\end{theorem}

Thus, in order to prove cosystolic expansion in examples, we should verify two things: local spectral expansion and coboundary expansion in the links. In our examples from \cite{KOconstruction} described below, local spectral expansion is already known and we are left with proving coboundary expansion for the links. In order to do so, we will develop machinery to prove coboundary expansion for symmetric complexes of a special type called coset complexes.

\subsection{Coboundary expansion for strongly symmetric simplicial complexes}

As noted above, unlike the case of graphs, in simplicial complexes a high dimensional version of Cheeger inequality does not hold. Thus, there is a need to develop machinery in order to prove coboundary expansion that does not rely on spectral arguments. For graphs such machinery is available, under the assumptions that the graph has a large symmetry group. A discussion regarding the Cheeger constant of symmetric graphs appear in \cite[Section 7.2]{Chung} and in particular, the following Theorem is proven there:

\begin{theorem}\cite[Theorem 7.1]{Chung}
\label{thm for symmetric graphs}
Let $X$ be a finite connected graph such that there is a group $G$ acting transitively on the edges of $X$. Denote $h(X)$ to be the Cheeger constant of $X$ and $D$ to be the diameter of $X$. Then $h(X) \geq \frac{1}{2D}$.
\end{theorem}

\begin{remark}
\label{no symmetry remark}
Note that the inequality stated in the Theorem does not hold without the assumption of symmetry. For instance, let $X_N$ by the graph that is the ball of radius $N$ in the $3$-regular infinite tree. Then the diameter of $X$ is $2N+1$ and $h (X_N)$ is of order $O(\frac{1}{2^N})$.
\end{remark}

In this paper, using the ideas of \cite{Grom} and \cite{LMM}, we prove a generalization of Theorem \ref{thm for symmetric graphs} to the setting of (strongly) symmetric simplicial complexes. We first define the notion of strongly symmetric simplicial complexes.
\begin{definition}[Strongly symmetric complex]
A simplicial complex $X$ is called {\em strongly symmetric} if there is a group that acts simply transitive on its top dimensional faces. E.g., For graphs (one dimensional complexes) we require a group that acts simply transitive on the edges.
\end{definition}

We then define a high dimensional notion of radius which we call a cone radius, but this definition is a little technical and thus omitted from the introduction (see Definition \ref{cone radius def}). We then prove the following:
\begin{theorem}[Informal, see Theorem \ref{Crad thm} for the formal statement]
\label{informal Crad thm intro}
Let $X$ be a strongly symmetric simplicial complex. If the $k$-dimensional (cone) radius of $X$ is bounded by $D$, then $\Exp^k_b (X) \geq \frac{1}{{n+1 \choose k+1} D}$, i.e., the $k$-coboundary expansion is bounded from below as a function of the $k$-th radius.
\end{theorem}

\subsection{Bounding the high dimensional radius for coset complexes}

By Theorem \ref{informal Crad thm intro}, in order to prove coboundary expansion for strongly symmetric complexes, it is enough to bound their high dimensional radius. Following the ideas of Gromov \cite{Grom}, we bound the radius by bounding certain filling constants, that we will not define here. In order to bound these filling constants and thus the high dimensional radius, we will assume that our strongly symmetric complex is of a special type, namely that it is a coset complex:
\begin{definition}[Coset complex]
\label{coset complex def}
Given a group $G$ with subgroups $K_{\lbrace i \rbrace}, i \in \I$, where $\I$ is a finite set. The \textit{coset complex} $X=X(G, (K_{\lbrace i \rbrace})_{i \in \I})$ is a simplicial complex defined as follows:
\begin{enumerate}
\item The vertex set of $X$ is composed of disjoint sets $S_i = \lbrace g K_{\lbrace i \rbrace} : g \in G \rbrace$.
\item For two vertices $g K_{\lbrace i \rbrace}, g' K_{\lbrace j \rbrace}$ where $i,j \in \I, g,g' \in G$, $\lbrace g K_{\lbrace i \rbrace}, g' K_{\lbrace j \rbrace} \rbrace \in X(1)$ if $i \neq j$ and  $g K_{\lbrace i \rbrace} \cap g' K_{\lbrace j \rbrace} \neq \emptyset$.
\item The simplicial complex $X$ is the clique complex spanned by the $1$-skeleton defined above, i.e., $\lbrace g_0 K_{\lbrace i_0 \rbrace},..., g_k  K_{\lbrace i_k \rbrace} \rbrace \in X(k)$ if for every $0 \leq j,j' \leq k$, $g_j K_{\lbrace i_j \rbrace} \cap g_{j'} K_{\lbrace i_{j'} \rbrace} \neq \emptyset$.
\end{enumerate}
\end{definition}

Although this Definition may seem daunting at first, we note that it is very natural in examples. Namely, Proposition \ref{symmetric complexes and coset complexes prop} shows that under some mild assumptions, strongly symmetric simplicial complexes are actually coset complexes.

As noted above, for coset complexes, every property of the complex should be reflected in some way in its symmetry group and its subgroup structure. Following this philosophy, we prove that for coset complexes, the $0$-th and $1$-th dimensional coboundary expansion can be bounded using the presentation of the group from which the complex arose.

In order to describe our result, we recall some definitions from group theory. Given a group $G$, a generating set $S \subseteq G$ is a set of elements of $G$ such that every element in $G$ can be written as a finite product (or sum if $G$ is commutative) of elements of $S$. One can always take $S = G$, but usually one can make due with a smaller set. For example, for the group $G$ of addition of integers modulo $n$, $G = (\mathbb{Z} / n \mathbb{Z}, +)$, one can take $S = \lbrace \pm 1 \rbrace$. Given a group $G$ with a generating set $S$, a word with letters in $S$ is called trivial if it equal to the identity. For example, in $G = (\mathbb{Z} / n \mathbb{Z}, +)$ with $S = \lbrace \pm 1 \rbrace$, the words $1 + 1+ (-1) + (-1)$ and $1 + ...+1 \text{ (n summands)} = n \cdot 1$ are trivial.

We say that a group $G$ has a presentation $G= \langle S \vert R \rangle$, if $S$ is a generating set of $G$ and $R$ is a set of trivial words called relations such that every trivial word in $G$ can be written using the words in $R \cup \lbrace s s^{-1},s^{-1} s  : s \in S\rbrace$ (allowing products, conjugations and inverses). Again, one can always take $S = G \setminus \lbrace e \rbrace$ and $R$ to be the entire multiplication table of $G$, i.e., all the words of the form $g_1 g_2 g_3^{-1} =e$, where $g_1 g_2 = g_3$. However, in concrete examples, one can usually make due with fewer generators and relations. For example, for the group $G =(\mathbb{Z} / n \mathbb{Z}, +)$ it is sufficient to take $S = \lbrace \pm 1 \rbrace$ and the single relation $n \cdot 1$. We note that it is not always easy to determine if a set of relations gives a presentation of $G$.

Given a presentation $G = \langle S \vert R \rangle$, the Dehn function for this presentation is a function $\Dehn: \mathbb{N} \rightarrow \mathbb{N}$ such that $\Dehn (m)$ describes how many elements of $R \cup \lbrace s s^{-1},s^{-1} s  : s \in S\rbrace$ does one need to write a trivial word in $G$ of length $\leq m$ (for an exact definition see Definition \ref{Dehn func def}). With this terminology, we prove the following:

\begin{theorem}
\label{N_0 + N_1 bound thm intro}
Let $G$ be a finite group with subgroups $K_{\lbrace i \rbrace}, i \in \lbrace 0,1,2 \rbrace$. Denote $X = X(G, (K_{\lbrace i \rbrace})_{i \in \lbrace 0,1,2 \rbrace})$. Assume that $G$ acts transitively on $X(2)$.

For every $i \in \lbrace 0,1,2 \rbrace$, denote $R_{i}$ to be all the non-trivial relations in the multiplication table of $K_{\lbrace i \rbrace}$, i.e., all the relations of the form $g_1 g_2 g_3 =e$, where $g_1,g_2, g_3 \in K_{\lbrace i \rbrace} \setminus \lbrace e \rbrace$. Assume that $G = \langle \bigcup_{i} K_{\lbrace i \rbrace} \vert \bigcup_{i} R_i \rangle$ and let $\Dehn$ denote the Dehn function of this presentation.

Then:
\begin{enumerate}
\item For
$$N_0 ' = 1+\max_{g \in G} \min \left\lbrace l : g = g_1...g_l \text{ and } g_1,...,g_l \in \bigcup_i K_{\lbrace i \rbrace} \right\rbrace,$$
it holds that $\Exp^0_b (X) \geq \frac{1}{3 N_0 '} $.
\item There is a universal polynomial $p(x,y)$ independent of $X$ such that
$$\Exp^1_b (X) \geq \frac{1}{3 p(2 N_0 '  + 1, \Dehn (2 N_0 ' +1))}.$$
\end{enumerate}
\end{theorem}

\subsection{Our construction}

So far, we described general tools that we developed in order to prove coboundary and cosystolic expansion. Now we will describe our construction from \cite{KOconstruction} on which we aim to apply these tools.

In \cite{KOconstruction}, we used coset complexes to construct $n$-dimensional spectral expanders. Below, we only describe the construction for $n=3$:
Fix $s \in \mathbb{N}, s > 4$ and $q$ be a prime power. Denote $G^{(s)}_q$ to be the group of $4 \times 4$ matrices with entries in $\mathbb{F}_q [t] / \langle t^s \rangle$ generated by the set
$$\lbrace e_{1,2} (a +bt), e_{2,3} (a +bt), e_{3,4} (a +bt), e_{4,1} (a +bt)  : a,b \in \mathbb{F}_q \rbrace.$$
For $0 \leq i \leq 2$, define $H_{\lbrace i \rbrace}$ to be the subgroup of $G^{(s)}_q$ generated by
$$\lbrace e_{j,j+1} (a +bt),  e_{4,1} (a +bt) : a,b \in \mathbb{F}_q, 1 \leq j \leq 3, j \neq i+1 \rbrace$$
and define $H_{\lbrace 3 \rbrace}$ to be the subgroup of $G^{(s)}$ generated by
$$\lbrace e_{1,2} (a +bt), e_{2,3} (a +bt), e_{3,4} (a +bt)  : a,b \in \mathbb{F}_q \rbrace.$$
Denote $X^{(s)}_q = X(G^{(s)}_q, (H_{\lbrace i \rbrace})_{i \in \lbrace 0,...,3 \rbrace})$ to be the coset complex as defined above.

The main result of \cite{KOconstruction} applied to $\lbrace X^{(s)}_q \rbrace_{s >4 }$ above  can be summarized as follows:
\begin{enumerate}
\item The family $\lbrace X^{(s)}_q \rbrace_{s >4 }$ has bounded degree (that depends on $q$).
\item The number of vertices of $X^{(s)}_q$ tends to infinity with $s$.
\item For every $s$, $X^{(s)}_q$ is $\frac{1}{\sqrt{q} -3}$-local spectral expander.
\end{enumerate}

In light of Theorem \ref{N_0 + N_1 bound thm intro}, we will also need some facts regrading the links of $X^{(s)}_q$. We give the following explicit description of the links in our construction: We note that for every fixed $q$ it holds that there is a complex $X$ such that for every $s>4$ and every vertex $v \in X^{(s)}_q$, there is a coset complex denoted $X_{link, q}$ such that the link of $v$ is isomorphic to $X_{link, q}$ (all the links are isomorphic).

The complex $X_{link, q}$ can be described explicitly as follows: Denote the group $G_{link, q}$ to be a subgroup of $4 \times 4$ invertible matrices with entries in $\mathbb{F}_q [t]$ in generated by the set
$\lbrace e_{i,i+1} (a +bt) : a,b \in \mathbb{F}_q, 1 \leq i \leq 4 \rbrace$. More explicitly, an $4 \times 4$ matrix $A$ is in $G_{link, q}$ if and only if
$$A (i,j) = \begin{cases}
1 & i=j \\
0 & i>j \\
a_0 + a_1 t + ... + a_{j-i} t^{j-i} & i<j, a_0,...,a_{j-i} \in \mathbb{F}_q
\end{cases},$$
(observe that all the matrices in $G$ are upper triangular).

For $0 \leq i \leq 3$, define a subgroup $K_{\lbrace i \rbrace} <G$ as
$$K_{\lbrace i \rbrace} = \langle e_{j, j+1} (a+bt) : j \in \lbrace 1,...,4 \rbrace \setminus \lbrace i+1 \rbrace, a,b \in \mathbb{F}_q \rangle.$$
Define $X_{link, q}$ to be the coset complex $X_{link, q}=X(G_{link, q}, (K_{\lbrace i \rbrace})_{i \in \lbrace 0,1,2 \rbrace})$. As noted above, for every $s >4$, all the $2$-dimensional links of $X^{(s)}_q$ are isomorphic to $X_{link, q}$. Also,
\begin{theorem}\cite[Theorems 2.4, 3.5]{KOCosetGeom}
The complex $X_{link, q}$ above is strongly symmetric, namely the group $G_{link, q}$ of unipotent matrices described above acts transitively on the triangles of $X_{link, q}$.
\end{theorem}

\subsection{New coboundary and cosystolic expanders}
Finally, we describe how the general machinery we developed can be applied in our construction.

First, by applying Theorem \ref{EK criterion thm} on the family $\lbrace X^{(s)}_q \rbrace_{s \in \mathbb{N}}$ yields the following Corollary:
\begin{corollary}
\label{EK intro coro}
Let $\lbrace X^{(s)}_q \rbrace_{s \in \mathbb{N}}$ be the family of $n$-dimensional simplicial complexes from \cite{KOconstruction}. Assume there is a constant $\varepsilon ' >0$ such that for every odd $q$, every $s$, every $0 \leq k \leq n-2$ and every $\tau \in X(k)$, $X_\tau$ is a $\varepsilon '$-coboundary expander.
Denote $Y^{(s)}_q$ to be the $(n-1)$-skeleton of $X^{(s)}_q$. Then for any sufficiently large odd prime power $q$, the family $\lbrace Y^{(s)}_q \rbrace_{s \in \mathbb{N}}$  is a family of bounded degree cosystolic expanders.
\end{corollary}

Thus, by this Corollary, in order to prove Theorem \ref{new cosystolic expanders thm intro informal} it is enough to show that for every odd $q$, there is a constant $\varepsilon ' >0$ such that for every odd $q$ and every $s \in \mathbb{N}$, the $2$-skeleton of the link of every vertex $v$ in $X^{(s)}_q$ is a $\varepsilon '$ coboundary expander.

As we noted, the links are strongly transitive coset complexes which we denoted $X_{q, link}$ and described explicitly above. By Theorem \ref{N_0 + N_1 bound thm intro}, in order to bound the coboundary expansion of the links, we need to consider the presentation of their symmetry group $G_{link, q}$ defined above. Generalizing on the work of Biss and Dasgupta \cite{BissD} we prove the following:
\begin{theorem}
\label{presentation thm into}
For any prime power $q$ denote $G_{q, link}, K_{\lbrace 0 \rbrace},K_{\lbrace 1 \rbrace}, K_{\lbrace 2 \rbrace}$ as above and for every $i \in \lbrace 0,1,2 \rbrace$, denote $R_{i}$ to be all the non-trivial relations in the multiplication table of $K_{\lbrace i \rbrace}$. Then
\begin{itemize}
\item
$$\sup_{q \text{ odd prime power}} \left( \max_{g \in G_{q, link}} \min \left\lbrace l : g = g_1...g_l \text{ and } g_1,...,g_l \in \bigcup_i K_{\lbrace i \rbrace} \right\rbrace \right) < \infty.$$
\item For every odd $q$ it holds that
$$G_{q, link} = \langle \bigcup_i K_{\lbrace i \rbrace} \vert \bigcup_i R_i \rangle$$
and the Dehn function of this presentation is bounded independently of $q$.
\end{itemize}
\end{theorem}

This Theorem combined with Theorem \ref{N_0 + N_1 bound thm intro} gives:
\begin{theorem}[First Main Theorem - new explicit two dimensional coboundary expanders]
\label{new coboundary expanders thm intro}
For every odd prime power $q$, $X_{link,q}$ is a coboundary expander and $\Exp_0 (X), \Exp_1 (X)$ are bounded from below by a constant that is independent of $q$.
\end{theorem}

Applying Corollary \ref{EK intro coro} it follows that:

\begin{theorem}[Second Main Theorem - elementary two dimensional bounded degree cosytolic expanders]
\label{new cosystolic expanders thm intro}
Let $s \in \mathbb{N}, s>4$ and $q$ be a prime power and $X^{(s)}_q$ as above. For every $s$, let $Y^{(s)}_q$ be the $2$-skeleton of $X^{(s)}_q$, i.e., the $2$-dimensional complex $Y^{(s)}_q= X^{(s)}_q (0) \cup X^{(s)}_q (1) \cup X^{(s)}_q (2)$.  For any sufficiently large odd prime power $q$, the family $\lbrace Y^{(s)}_q \rbrace_{s \in \mathbb{N}, s >4}$ is a family of bounded degree cosystolic expanders.
\end{theorem}

\subsection{Organization of the paper}
This paper is organized as follows: In Section \ref{Homological and Cohomological definitions and notations sec}, we review the basic definitions and notations regarding (co)homology that we will use throughout the paper. In Section \ref{Cone radius as a bound on coboundary expansion sec}, we prove that for symmetric simplicial complexes, the cone radius can be used to bound the coboundary expansion. In Section \ref{Bounding the high order radius by the filling constants sec}, we define filling constants of a simplicial complex and show that filling constants of the complex can be used to bound the cone radius. In Section \ref{Symmetric complexes and coset complexes sec}, we review the idea of coset complexes and show that our assumption of strong symmetry combined with some extra assumptions on our complex imply that it is a coset complex. In Section \ref{Bounding the first two filling constants for coset complexes sec}, we deduce a bound on the first two filling constants for a coset complex in terms of algebraic properties of the presentation of the group and subgroups from which it arises. In Section \ref{New coboundary expanders sec}, we give new examples of coboundaries expanders arising from coset complexes of unipotent groups. In Section \ref{New cosys expanders sec}, we give new examples of bounded degree cosystolic and topological expanders. Last, in Appendix \ref{The existence of a cone function and vanishing of (co)homology sec}, we show that the existence of a cone function is equivalent to the vanishing of (co)homology.

\subsection{Acknowledgement}

The first named author was supported by ERC and BSF. The second named author was supported by ISF (grant No. 293/18).

\section{Homological and Cohomological definitions and notations}
\label{Homological and Cohomological definitions and notations sec}

The aim of this section is to recall a few basic definitions regarding homology and cohomology of simplicial complexes that we will need below.

Let $X$ be an $n$-dimensional simplicial complex. A simplicial complex $X$ is called {\em pure} if every face in $X$ is contained in some face of size $n+1$.  The set of all $k$-faces of $X$ is denoted $X(k)$, and we will be using the convention in which $X(-1) = \{\emptyset\}$.

We denote by $C_k (X) = C_k(X, \mathbb{F}_2)$ the $\mathbb{F}_2$-vector space with basis $X(k)$ (or equivalently, the $\mathbb{F}_2$-vector space of subsets of $X(k)$), and $C^k (X) = C^k(X,\mathbb{F}_2)$ the $\mathbb{F}_2$-vector space of functions from $X(k)$ to $\mathbb{F}_2$.

The {\em boundary map} $\partial_k : C_k(X,\mathbb{F}_2) \rightarrow C_{k-1}(X,\mathbb{F}_2)$ is:
\begin{equation*}
\partial_k(\sigma) = \sum_{\tau \subset \sigma, |\tau|=|\sigma|-1} \tau,
\end{equation*}
where $\sigma \in X(k)$, and the {\em coboundary map} $d_k: C^k(X,\mathbb{F}_2) \rightarrow C^{k+1}(X,\mathbb{F}_2)$ is:
\begin{equation*}
d_k( \phi)(\sigma)=\sum_{\tau \subset \sigma, |\tau|=|\sigma|-1} \phi (\tau),
\end{equation*}
where $\phi \in C^k$ and $\sigma \in X(k+1)$.

For $A \in C_k (X)$ and $\phi \in C^k (X)$, we denote
\begin{equation*}
\phi (A) = \sum_{\tau \in A} \phi (\tau),
\end{equation*}

Thus, for $\phi \in C^k(X)$ and $A \in C_{k+1} (X)$

\begin{equation*}
(d_k \phi) (A) = \phi (\partial_{k+1} A)
\end{equation*}

We sometimes refer to $k$-chains as subsets of $X(k)$, e.g., the $0$-chain $\lbrace u \rbrace + \lbrace v \rbrace$ will be sometimes referred to as the set $\lbrace \lbrace u \rbrace, \lbrace v \rbrace \rbrace$ . For $A \in C_k (X)$, we denote $\vert A \vert$ to be the size of $A$ as a set.

Well known and easily calculated equations are:
\begin{eqnarray}\label{eqnarray-double-partial}
\partial_k \circ \partial_{k+1}=0 \mbox{ and } d_{k+1} \circ d_{k} = 0
\end{eqnarray}

Thus, if we denote:
$B_k (X) =B_k(X,\mathbb{F}_2) = \mbox{Image}(\partial_{k+1})=$\mbox{ the space of $k$-boundaries}.

$Z_k (X) =Z_k(X,\mathbb{F}_2) = \mbox{Ker}(\partial_{k+1})=$\mbox{ the space of $k$-cycles}.

$B^k (X) = B^k(X,\mathbb{F}_2) = \mbox{Image}(d_{k-1})=$\mbox{ the space of $k$-coboundaries}.

$Z^k (X) =Z^k(X,\mathbb{F}_2) = \mbox{Ker}(d_{k})=$\mbox{ the space of $k$-cocycles}.

We get from (\ref{eqnarray-double-partial})
\begin{equation*}
B_k (X) \subseteq Z_k (X) \subseteq C_k (X) \mbox{ and } B^k (X) \subseteq Z^k (X) \subseteq C^k (X).
\end{equation*}

Define the quotient spaces $\widetilde{H}_k(X) =Z_k (X)/B_k (X)$ and $\widetilde{H}^k(X) =Z^k (X) /B^k (X)$, the $k$-homology and the $k$-cohomology groups of $X$ (with coefficients in $\mathbb{F}_2$).

\section{Cone radius as a bound on coboundary expansion}
\label{Cone radius as a bound on coboundary expansion sec}

Below, we define a generalized notion of diameter (or more precisely radius) of a simplicial complex. We will later show that in symmetric simplicial complexes a bound on this radius yields a bound on the coboundary expansion of the complex.
\begin{definition}[Cone function]
Let $X$ be a pure $n$-dimensional simplicial complex. Let $-1 \leq k \leq n-1$ be a constant and $v$ be a vertex of $X$. A $k$-cone function with apex $v$ is a linear function $\Cone_{k}^v: \bigoplus_{j=-1}^k C_{j} (X) \rightarrow \bigoplus_{j=-1}^k C_{j+1} (X)$ defined inductively as follows:
\begin{enumerate}
\item For $k=-1$, $\Cone_{-1}^v (\emptyset) = \lbrace v \rbrace$.
\item For $k \geq 0$,  $\left. \Cone_{k}^v \right\vert_{\bigoplus_{j=-1}^{k-1} C_{j} (X)}$ is a $(k-1)$-cone function with an apex $v$ and for every $A \in C_k (X)$, $\Cone_{k}^v (A) \in C^{k+1} (X)$ is a $(k+1)$-chain that fulfills the equation
$$\partial_{k+1} \Cone_{k}^v (A) = A + \Cone_{k}^v (\partial_k A).$$
\end{enumerate}
\end{definition}

\begin{observation}
\label{Cone function defined on simplices obs}
By linearity, the condition that
$$\partial_{k+1} \Cone_{k}^v (A) = A + \Cone_{k}^v (\partial_k A), \forall A \in C^k (X)$$
is equivalent to the condition:
$$\partial_{k+1} \Cone_{k}^v (\tau) = \tau + \Cone_{k}^v (\partial_k \tau), \forall \tau \in X(k).$$
\end{observation}

\begin{remark}
We note that by linearity, a $k$-cone function is needs only to be defined on $k$-simplices, but it gives us homological fillings for every $k$-cycle in $X$: for every $A \in Z_k (X)$,
$$\partial_{k+1} \Cone_{k}^v (A) = A + \Cone_{k}^v (\partial_k A) =  A + \Cone_{k}^v (0) = A,$$
i.e., $\partial_{k+1} \Cone_{k}^v (A) = A$. This might be computationally beneficial for other needs (apart from the results of this paper), since usually there are exponentially more $k$-cycles than $k$-simplices.
\end{remark}

\begin{example}[$0$-cone example]
\label{0-Cone example}
Let $X$ be an $n$-dimensional simplicial complex. Fix some vertex $v$ in $X$. By definition, for every $\lbrace u \rbrace \in X(0)$, $\Cone_{0}^v (\lbrace u \rbrace)$ is a $1$-chain such that $\partial_0 \Cone_{0}^v (\lbrace u \rbrace) = \lbrace u \rbrace + \lbrace v \rbrace$.

If the $1$-skeleton of $X$ is connected, we can define $\Cone_{0}^v (\lbrace u \rbrace)$ to be a $1$-chain that consists of a sum of edges that form a path between $\lbrace u \rbrace$ and $\lbrace v \rbrace$. If the $1$-skeleton of $X$ is not connected, a $0$-cone function does not exist: for $\lbrace u \rbrace \in X(0)$ that is not in the connected component of $\lbrace v \rbrace$, $\Cone_{0}^v (\lbrace u \rbrace)$ cannot be defined. Assuming that the $1$-skeleton of $X$ is connected, we note that the construction of $\Cone_{0}^v$ is usually not unique: different choices of paths between $\lbrace u \rbrace$ and $\lbrace v \rbrace$ give different $0$-cone functions.
\end{example}

\begin{example}[$1$-cone example]
\label{1-Cone example}
Let $X$ be an $n$-dimensional simplicial complex. Assume that the $1$-skeleton of $X$ is connected and define a $0$-cone function as in the example above and define $\Cone_{2}^v$ on $C_{0} (X)$ as that $0$-cone function. We note that for every $\lbrace u,w \rbrace \in X(1)$, $\lbrace u,w \rbrace + \Cone_{0}^v (\lbrace u \rbrace) + \Cone_{0}^v (\lbrace w \rbrace)$ forms a closed path, i.e., a $1$-cycle, in $X$. If $\widetilde{H}_1 (X) = 0$, we can deduce that $\lbrace u,w \rbrace + \Cone_{1}^v (\lbrace u \rbrace) + \Cone_{1}^v (\lbrace w \rbrace)$ is a boundary. Therefore, for every $\lbrace u,w \rbrace \in X(1)$, we can choose $\Cone_{1}^v (\lbrace u,w \rbrace) \in X(2)$ such that
$$\partial_2 \Cone_{1}^v = \lbrace u,w \rbrace + \Cone_{1}^v (\lbrace u \rbrace) + \Cone_{1}^v (\lbrace w \rbrace).$$
\end{example}

\begin{definition}[Cone radius]
\label{cone radius def}
Let $X$ be an $n$-dimensional simplicial complex, $-1 \leq k \leq n-1$ and $v$ a vertex of $X$. Given a $k$-cone function $\Cone_{k}^v$ define the volume of $\Cone_{k}^v$ as
$$\Vol (\Cone_{k}^v) = \max_{\tau \in X(k)} \vert \Cone_{k}^v (\tau) \vert.$$

Define the $k$-th cone radius of $X$ to be
$$\Crad_k (X) = \min \lbrace \Vol (\Cone_{k}^v) : \lbrace v \rbrace \in X(0), \Cone_{k}^v \text{ is a k-cone function} \rbrace.$$
If $k$-cone functions do not exist, we define $\Crad (X) = \infty$.
\end{definition}

\begin{remark}
The reason for the name ``cone radius'' is that in the case where $k=0$, $\Crad_0 (X)$ is exactly the (graph) radius of the $1$-skeleton of $X$.
Indeed, for $k=0$, choose $\lbrace v \rbrace \in X (0)$ such that for every $\lbrace v' \rbrace \in X (0)$,
$$\max_{\lbrace u \rbrace \in X(0)} \dist (v,u) \leq \max_{\lbrace u \rbrace \in X(0)} \dist (v',u),$$
where $\dist$ denotes the path distance. For such a $\lbrace v \rbrace \in X(0)$, define $\Cone_0^v (\lbrace u \rbrace)$ to be the edges of a shortest path between $v$ and $u$. By our choice of $v$, it follows that $\Vol (\Cone_0^v)$ is the radius of the one-skeleton of $X$ and we leave it to the reader to verify that this choice gives $\Crad_0 (X) = \Vol (\Cone_0^v)$.
\end{remark}

The main result of this section is that in a symmetric simplicial complex $X$, the $k$-th cone radius gives a lower bound on  $\Exp^k_b (X)$:
\begin{theorem}
\label{Crad thm}
Let $X$ be a pure finite $n$-dimensional simplicial complex. Assume that $X$ is strongly symmetric, i.e., that there is a group $G$ of automorphisms of $X$ acting transitively on $X(n)$. For every $0 \leq k \leq n-1$, if $\Crad_k (X) < \infty$, then $\Exp^k_b (X) \geq \frac{1}{{n+1 \choose k+1} \Crad_k (X)}$.
\end{theorem}

Theorem \ref{Crad thm} stated above generalizes a result of of Lubotzky, Meshulam and Mozes \cite{LMM} in which coboundary expansion was proven for symmetric simplicial complexes given that they are ``building-like'', i.e., that they have have sub-complexes that behave (in some sense) as apartments in a Bruhat-Tits building.

We note that the notion of a cone function is already evident in Gromov's original work \cite{Grom}. Gromov considered what he called ``random cones'', which was a probability over a family of cone functions and show that the expectancy of the occurrence of a simplex in the support of this family bounds $\Exp^k_b (X)$ (see also the work of Kozlv and Meshulam \cite[Theorem 2.5]{KozM}). Using Gromov's terminology, in the proof of the Theorem above, we show that under the assumption of symmetry a single cone function yields a family of random cones and the needed expectancy is bounded by the cone radius. In the sake of completeness, we will not prove the Theorem without using Gromov's results.

In order to prove Theorem \ref{Crad thm}, we will need some additional lemmas.
\begin{lemma}
\label{lemma iota}
For $-1 \leq k \leq n-1$ and a $k$-cone function $\Cone_{k}^v$ with apex $v$. Define the contraction operator $\iota_{\Cone_{k}^v}$,
$$\iota_{\Cone_{k}^v} :\bigoplus_{j=-1}^k C^{j+1} (X) \rightarrow \bigoplus_{j=-1}^k C^{j} (X)$$
as follows: for $\phi \in C^{j+1} (X)$ and $A \in C_{j} (X)$, we define
$$(\iota_{\Cone_{k}^v} \phi) (A) = \phi (\Cone_{k}^v (A) ).$$
Then for every $\phi \in C^k (X)$,
$$\iota_{\Cone_{k}^v} d_{k} \phi = \phi + d_{k-1}  \iota_{\Cone_{k}^v} \phi.$$
\end{lemma}

\begin{proof}
Let $A \in C_{k} (X)$, then
\begin{dmath*}
\iota_{\Cone_{k}^v} d_{k} \phi (A)=
(d_{k} \phi) (\Cone_{k}^v (A)) =
\phi (\partial_{k+1} (\Cone_{k}^v (A))) =
\phi (A +  (\Cone_{k}^v (\partial_{k} A))) =
\phi (A) + (\iota_{\Cone_{k}^v} \phi) (\partial_{k} A) =
\phi (A) + (d_{k-1} \iota_{\Cone_{k}^v} \phi) (A),
\end{dmath*}
as needed.
\end{proof}

Naively, it might seem that this Lemma gives a direct approach towards bounding the coboundary expansion: if one could find is some constant $C = C(n,k, \Crad_k (X))$ such that $w(\iota_{\Cone_{k}^v} d_{k} \phi) \leq C w (d_k \phi)$, then for every $\phi$,
$$\frac{w (d_k \phi)}{\min_{\psi \in B^k (X)} w(\phi + \psi)} \geq \frac{1}{C} \frac{w (\iota_{\Cone_{k}^v} d_{k} \phi)}{w(\phi + d_{k-1}  \iota_{\Cone_{k}^v} \phi)} = \frac{1}{C}.$$
However, by Remark \ref{no symmetry remark}, we note that without symmetry, the existence of a $k$-cone function cannot give an effective bound on the coboundary expansion.

Our proof strategy below is to improve on this naive idea by using the symmetry of $X$: we will show that for a group $G$ that acts on $X$, the group $G$ also acts on $k$-cone functions and we will denote this action by $\rho$. We then show that when $G$ acts transitively on $X(n)$, we can average the action on the $k$-cone function that realizes the cone radius and deduce that
$$\frac{1}{\vert G \vert} \sum_{g \in G} w(\iota_{\rho (g).\Cone_{k}^v} d_{k} \phi) \leq {n+1 \choose k+1} \Crad_k (X)  w (d_k \phi).$$
Thus, using an averaged version of the naive argument above will get a bound on the coboundary expansion.

We start by defining an action on $k$-cone functions. Assume that $G$ is a group acting simplicially on $X$. For every $g \in G$ and every $k$-cone function $\Cone_k^v$ define
$$(\rho (g).\Cone_k^v) (A) = g.(\Cone_k^v (g^{-1}.A)), \forall A \in \bigoplus_{j=-1}^k C_{j} (X).$$

\begin{lemma}
\label{group action on cone func lemma}
For $g \in G$, $-1 \leq k \leq n-1$ and a $k$-cone function $\Cone_{k}^v$ with apex $v$, $\rho (g). \Cone_{k}^v$ is a $k$-cone function with apex $g.v$ and $\Vol (g.\Cone_{k}^v) = \Vol (\Cone_{k}^v)$. Moreover, $\rho$ defines an action of $G$ on the set of $k$-cone functions.
\end{lemma}

\begin{proof}
If we show that $g.\Cone_{k}^v$ is a $k$-cone function the fact that $\Vol (g.\Cone_{k}^v) = \Vol (\Cone_{k}^v)$ will follow directly from the fact that $G$ acts simplicially.

The proof that $\rho(g).\Cone_{k}^v$ is a $k$-cone function is by induction on $k$. For $k=-1$,
$$(\rho(g).\Cone_{v,-1}) (\emptyset) = g. \Cone_{v,-1} (g^{-1}. \emptyset) = g.\Cone_{v,-1} (\emptyset) = g.  \lbrace v \rbrace = \lbrace g.v\rbrace,$$
then $\rho(g).\Cone_{v,-1}$ is a $(-1)$-cone function with an apex $g.v$.

Assume the assertion of the lemma holds for $k-1$. Thus, $\left. \rho(g).\Cone_{k}^v \right\vert_{\bigoplus_{j=-1}^{k-1} C_{j} (X)}$ is a $(k-1)$-cone function with an apex $g.v$ and, by Observation \ref{Cone function defined on simplices obs}, we are left to check that for every $\tau \in X(k)$,
$$\partial_{k+1} (\rho(g).\Cone_{k}^v) (\tau) = \tau + \rho(g).\Cone_{k}^v (\partial_k  \tau).$$

Note that the $G$ acts simplicially on $X$ and thus the action of $G$ commutes with the $\partial$ operator. Therefore, for every $\tau \in X(k)$,
\begin{dmath*}
\partial_{k+1} (\rho(g).\Cone_{k}^v) (\tau) = \partial_{k+1} (g.(\Cone_{k}^v (g^{-1}.\tau))) = g. \left(\partial_{k+1} \Cone_{k}^v (g^{-1}.\tau) \right) =
g.\left(g^{-1}.\tau + \Cone_{k}^v (\partial_k  g^{-1}.\tau) \right) =
\tau +  g.\Cone_{k}^v (g^{-1}. \partial_k\tau)  =
\tau + \rho(g).\Cone_{k}^v (\partial_k  \tau).
\end{dmath*}

The fact that $\rho$ is an action is straight-forward and left for the reader.
\end{proof}

Applying our proof strategy above, will lead us to consider the constant $\theta (\eta)$ defined in the Lemma below. 

\begin{lemma}
\label{theta lemma}
Assume that $G$ is a group acting simplicially on $X$ and that this action is transitive on $n$-simplices.
Let $0 \leq k \leq n-1$ and assume that $\Crad_k (X) < \infty$. Fix $\Cone_{k}^v$ to be a $k$-cone function such that  $\Crad_k (X) =\Vol (\Cone_{k}^v)$. For every $\eta \in X(k+1)$, denote
$$\theta (\eta) = \frac{1}{w (\eta) \vert G \vert} \sum_{g \in G} \sum \lbrace w (\tau) : \tau \in X(k), \eta \in (\rho (g). \Cone_{k}^v) (\tau))\rbrace.$$
Then for every $\eta \in X(k+1)$, $\theta (\eta) \leq {n+1 \choose k+1} \Crad_k (X)$.
\end{lemma}

\begin{proof}
Fix some $\eta \in X(k+1)$. First, we note that $G$ acts transitively on $X(n)$ and therefore $\bigcup_{g G_\eta} \lbrace g. \sigma : \sigma \in X(n), \eta \subseteq \sigma \rbrace = X(n)$. This yields that
$$\vert X(n) \vert \leq \frac{\vert G \vert}{\vert G_\eta \vert} \vert \lbrace \sigma \in X(n) : \eta \subseteq \sigma \rbrace \vert,$$
and therefore
$\vert G_\eta \vert \leq  \vert G \vert {n+1 \choose k+1} w(\eta)  .$

Second, we note that for every $g \in G$, and every $\eta \in X(k+1)$,
\begin{dmath*}
{\eta \in (\rho (g).\Cone_k^v) (\tau)} \Leftrightarrow
{\eta \in g.(\Cone_k^v (g^{-1}.\tau))} \Leftrightarrow
{g^{-1}.\eta \in \Cone_k^v (g^{-1}.\tau)}.
\end{dmath*}
Thus,
\begin{dmath*}
\theta (\eta) = {\frac{1}{w (\eta) \vert G \vert} \sum_{g \in G} \sum \lbrace w (\tau) : \tau \in X(k), \eta \in (\rho (g). \Cone_{k}^v) (\tau))\rbrace} =
{\frac{1}{w (\eta) \vert G \vert} \sum_{g \in G} \sum \lbrace w (\tau) : \tau \in X(k), g^{-1}.\eta \in \Cone_{k}^v (g^{-1}.\tau)\rbrace} =
{\frac{1}{w (\eta) \vert G \vert} \sum_{g \in G} \sum \lbrace w (g^{-1}.\tau) : \tau \in X(k), g^{-1}.\eta \in \Cone_{k}^v (g^{-1}.\tau)\rbrace} =
{\frac{1}{w (\eta) \vert G \vert} \sum_{g \in G} \sum \lbrace w (\tau) : \tau \in X(k), g^{-1}.\eta \in \Cone_{k}^v (\tau)\rbrace} =
{\frac{1}{w (\eta) \vert G \vert} \sum_{g \in G} \sum_{\tau \in X(k)} \sum_{g^{-1}.\eta \in \Cone_{k}^v (\tau)} w(\tau)} =
{\frac{1}{w (\eta)} \sum_{\tau \in X(k)} w(\tau) \sum_{g \in G}  \sum_{g^{-1}.\eta \in \Cone_{k}^v (\tau)} \frac{1}{\vert G \vert}} \leq
{\frac{1}{w (\eta)} \sum_{\tau \in X(k)} w(\tau)  \vert \Cone_{k}^v (\tau) \vert \frac{\vert G_\eta \vert}{\vert G \vert}} \leq
{\frac{1}{w (\eta)} \sum_{\tau \in X(k)} w(\tau)  \Crad_k (X) {n+1 \choose k+1} w(\eta)} =
{{n+1 \choose k+1} \Crad_k (X) \sum_{\tau \in X(k)} w(\tau) = {n+1 \choose k+1} \Crad_k (X)},
\end{dmath*}
as needed.
\end{proof}

We turn now to prove Theorem \ref{Crad thm}:
\begin{proof}
Assume that $G$ is a group acting simplicially on $X$ such that the action is transitive on $X(n)$. Let $0 \leq k \leq n-1$ and assume that $\Crad_k (X) < \infty$. For $\phi \in C^k (X)$, we denote
$$[ \phi ] = \lbrace \phi + \psi : \psi \in B^k (X) \rbrace \text{, and } w([ \phi ] )= \min_{\phi ' \in [\phi]} w(\phi ').$$
Thus, we need to prove that for every $\phi \in C^k (X)$,
$$\dfrac{1}{{n+1 \choose k+1} \Crad_k (X)} w ([ \phi ]) \leq w(d_{k} \phi),$$
or equivalently,
$$\vert G \vert w ([ \phi ]) \leq \vert G \vert w(d_{k} \phi) \left( {n+1 \choose k+1} \Crad_k (X) \right).$$

Fix $\Cone_k^v$ to be a $k$-cone function such that $\Crad_k (X) = \Vol (\Cone_k^v)$. By Lemma \ref{group action on cone func lemma}, for every $g \in G$, $\rho (g).\Cone_k^v$ is a $k$-cone function.

In the notation of Lemma \ref{lemma iota}, for every $g \in G$, denote $\iota_g = \iota_{\rho (g).\Cone_k^v}$. By Lemma \ref{lemma iota}, for every $g \in G$, $w ([ \phi ]) \leq w (\iota_g d_{k} \phi)$ and therefore
\begin{dmath*}
\vert G \vert w ([ \phi ]) \leq \sum_{g \in G} w(\iota_g d_{k} \phi ) =
\sum_{g \in G} \sum \left\lbrace w(\tau) : \tau \in \supp (\iota_g d_{k} \phi) \right\rbrace = \\
\sum_{g \in G} \sum \left\lbrace w(\tau) : \tau \in X(k), d_{k} \phi (\rho (g).\Cone_k^v (\tau)) =1 \right\rbrace \leq \\
\sum_{g \in G} \sum \left\lbrace w(\tau) : \tau \in X(k), \supp (d_{k} \phi) \cap \rho (g).\Cone_k^v (\tau) \neq \emptyset \right\rbrace \leq \\
\sum_{\eta \in \supp (d_{k} \phi)} \sum_{g \in G} \sum \lbrace w(\tau) : \tau \in X(k), \eta \in \rho (g).\Cone_k^v (\tau) \rbrace = \\
 \sum_{\eta \in \supp (d_{k} \phi)} \vert G \vert w (\eta) \theta (\eta) \leq^{\text{Lemma }\ref{theta lemma}}
\vert G \vert w(d_{k} \phi) \left({n+1 \choose k+1} \Crad_k (X) \right),
\end{dmath*}
as needed.
\end{proof}


The converse of Theorem \ref{Crad thm} is also true, i.e., the existence of a cone function is equivalent to vanishing of (co)homology and thus to coboundary expansion. This fact will not be used in the sequel and thus we give the exact statement and the proof in Appendix \ref{The existence of a cone function and vanishing of (co)homology sec}.

\section{Bounding the high order radius by the filling constants of the complex}
\label{Bounding the high order radius by the filling constants sec}
Once we realize that the $k$-th cone radius of the complex can be used to bound the generalized Cheeger constant of the complex, we need to find a way to bound the cone radius. In order to do so, we define what we call the ``filling constants'' of the complex. We will discuss to types of filling constants - homological and homotopical.

In a nutshell, the homological filling constant measure for a given $k$-cycle $B$, how large is a $k+1$-cochain $A$ that satisfy $\partial_{k+1} A=B$.  The filling constants will be small if $|A|$ is not much larger than $|B|$. They will be infinite if there is no $A$ such that $\partial_{k+1} A=B$.

In order to give the precise definition, we will need the following notation:
\begin{enumerate}
\item For every $0 \leq k \leq n-1$, $\Sys_k (X)$ denotes the size of the smallest $k$-systole in $X$:
$$\Sys_k (X) = \min \lbrace \vert A \vert : A \in Z_k (X) \setminus B_{k} (X)\rbrace,$$
(if $Z_k (X) = B_{k} (X)$, we define $\Sys_k (X) = \infty$).
\item For every $0 \leq k \leq n-1$ and $B \in Z_{k} (X)$, we define $\Fill_k (B)$ as follows:
for $B \in B_{k} (X)$,
$$\Fill_k (B) = \min \lbrace \vert A \vert : A \in C_{k+1} (X), \partial_{k+1} A= B \rbrace,$$
and for $B \in Z_{k} (X) \setminus B_k (X)$, $\Fill_k (B) = \infty$.
Furthermore, for every $M \in \mathbb{N}$, we define
$$\Fill_{k} (M) = \max \lbrace \Fill_k (B) : B \in Z_k(X), \vert B \vert \leq M \rbrace.$$
\end{enumerate}

\begin{proposition}
\label{filling constants prop}
Let $X$ be a pure finite $n$-dimensional simplicial complex.
Define the following sequence of constants recursively: $M_{-1} = 1$ and for every $0 \leq k \leq n-1$,
$$M_k  = \Fill_{k} ((k+1) M_{k-1}  +1).$$
If $\Sys_j (X) > (j+1) M_{j-1} +1$ for every $0 \leq j \leq k$, then $\Crad_k (X) \leq M_k$.
\end{proposition}

In order to prove this Proposition, we will need the following Lemma:

\begin{lemma}
\label{cone equation is cycle lemma}
Let $X$ be an $n$-dimensional simplicial complex, $0 \leq k \leq n-1$ and $\lbrace v \rbrace \in X(0)$. If $\Cone_{k-1}^v$ is a $(k-1)$-cone function, then for every $\tau \in X(k)$, $\tau + \Cone_{k-1}^v (\partial_k \tau) \in Z_k (X)$.
\end{lemma}

\begin{proof}
For $k=0$, we note that for every $\lbrace u \rbrace \in X(0)$,
$$\Cone_{-1}^v (\partial_0 \lbrace u \rbrace) = \Cone_{-1}^v (\emptyset) = \lbrace v \rbrace$$
and thus
$$\partial_0 (\lbrace u \rbrace + \Cone_{-1}^v (\partial_0 \lbrace u \rbrace)) = \partial_0 (\lbrace u \rbrace + \lbrace v \rbrace) = 2 \emptyset = 0.$$
Assume that $k >0$, then by the definition of the cone function, for every $\tau \in X(k-1)$,
\begin{dmath*}
\partial_k (\tau + \Cone_{k-1}^v (\partial_k \tau)) =
{\partial_k \tau + \partial_k \Cone_{k-1}^v (\partial_k \tau))} =
{\partial_k \tau + \partial_k \tau + \Cone_{k-1}^v (\partial_{k-1} \partial_k \tau))} =
2 \partial_k \tau + \Cone_{k-1}^v (0) = 0.
\end{dmath*}
\end{proof}

\begin{proof}[Proof of Proposition \ref{filling constants prop}]
Let $M_k, 0 \leq k \leq n-1$ be the constants of Proposition \ref{filling constants prop}. Fix $ 0 \leq k \leq n-1$ and assume that for every $0 \leq j \leq k$, $\Sys_j (X) > (j+1) M_{j-1} +1$. We will show that under these conditions $\Crad_k (X) \leq M_k$.

The proof is by an inductive construction of a $j$-cone function $\Cone_j^v$ with volume $\leq M_j$ for every $-1 \leq j \leq k$. The construction is as follows: fix some $\lbrace v \rbrace \in X(0)$ and define $\Cone_{-1}^v (\emptyset) = \lbrace v \rbrace$. Then $\Vol (\Cone_{-1}^v) = 1 = M_{-1}$ as needed.

Let $0 \leq j \leq k$ and assume that $\Cone_{j-1}^v$ is defined such that $\Vol (\Cone_{j-1}^v) \leq M_{j-1}$. We define $\left. \Cone_{j}^v \right\vert_{\bigoplus_{l=-1}^{j-1} C_l (X)} = \Cone_{j-1}^v$ and we are left to define $\Cone_{j}^v (\tau)$ for every $\tau \in X(j)$. Fix some $\tau \in X(j)$. By our induction assumption,
\begin{dmath*}
\vert \tau + \Cone_{j-1}^v (\partial_j \tau) \vert \leq \\
{1 + \sum_{\alpha \in X(j-1), \alpha \subseteq \tau} \vert \Cone_{j-1}^v (\alpha) \vert} \leq \\
{1+(j+1) \Vol (\Cone_{j-1}^v)} \leq (j+1) M_{j-1} +1.
\end{dmath*}
We assumed that $\Sys_j (X) > (j+1) M_{j-1} +1$ and thus $\tau + \Cone_{j-1}^v (\partial_j \tau) \notin Z_j (X) \setminus B_j (X)$.
By Lemma \ref{cone equation is cycle lemma}, for any $\tau \in X(j)$, $\tau + \Cone_{j-1}^v (\partial_j \tau) \in Z_j (X)$ and therefore we deduce that $\tau + \Cone_{j-1}^v (\partial_j \tau) \in B_j (X)$, i.e., there is $A \in C_{j+1} (X)$ such that
$$\partial_{j+1} A = \tau + \Cone_{j-1}^v (\partial_j \tau).$$
Also, we can choose this $A$ to be minimal in the sense that for every $A' \in C_{k+1} (X)$, if
$$\partial_{j+1} A' = \tau + \Cone_{j-1}^v (\partial_j \tau),$$
then $\vert A \vert \leq \vert A' \vert$. For such a minimal $A$, define $\Cone_j^v (\tau) = A$. Since we chose $A$ to be minimal, it follows that
\begin{dmath*}
\vert \Cone_j^v (\tau) \vert \leq \Fill_j (\tau + \Cone_{j-1}^v (\partial_j \tau)) \leq \\
\Fill_j (\vert \tau + \Cone_{j-1}^v (\partial_j \tau) \vert) \leq
\Fill_j ((j+1) M_{j-1} +1 ) = M_j.
\end{dmath*}
Thus, $\Vol (\Cone_j^v) \leq M_j$ as needed.
\end{proof}

Combining Proposition \ref{filling constants prop} with Theorem \ref{Crad thm}, we deduce the following:

\begin{theorem}
\label{filling constants criterion for coboundary exp thm1}
Let $X$ be a pure finite $n$-dimensional simplicial complex. Let $M_k$ be the constants defined in Proposition \ref{filling constants prop}. Assume that $X$ is strongly symmetric, i.e., that there is a group $G$ of automorphisms of $X$ acting transitively on $X(n)$. Fix $0 \leq k \leq n-1$. If $\Sys_j (X) > (j+1) M_{j-1} +1$ for every $0 \leq j \leq k$, then $\Exp^k_b (X) \geq \frac{1}{{n+1 \choose k+1} M_k}$.
\end{theorem}

A variant of Proposition \ref{filling constants prop} that will be useful for use is working with homotopy fillings instead of homological filling. In order to define these constants, we will need some additional notations. For $k \in \mathbb{N}$ denote $D^k$ to be the unit disk in $\mathbb{R}^{k}$, i.e.,
$$D^k = \lbrace \overline{x} \in \mathbb{R}^{k} : \Vert \overline{x} \Vert \leq 1 \rbrace,$$
where $\Vert . \Vert$ here denotes the Euclidean norm in $\mathbb{R}^{k}$. We work with the convention that $D^0$ is always a single point. Also denote $S^k$ to be the unit sphere in $\mathbb{R}^{k+1}$, i.e.,
$$S^k = \lbrace \overline{x} \in \mathbb{R}^{k+1} : \Vert \overline{x} \Vert = 1 \rbrace.$$
We work with the convention that $S^{-1}$ is always the empty set. We further denote $S^k_\triangle$ and $D^k_\triangle$ to be triangulations of $S^k$ and $D^k$ respectively. We treat $S^k_\triangle$ (and $D^k_\triangle$) as a pure $k$-dimensional ($(k+1)$-dimensional) simplicial complex. Note that for $k>0$, $S^k_\triangle, D^k_\triangle$ are not well defined, since there are infinitely many triangulations of $S^k$ and $D^k$. With these notations, we can define the notion of $k$-connectedness:
\begin{definition}
A simplicial complex $X$ is called $k$-connected if for every $-1 \leq i \leq k$, if there is a simplicial function $f: S^i_\triangle \rightarrow X$ (where $S^i_\triangle$ is some triangulated $i$-sphere), then there is a triangulated $i+1$ disc $D^{i+1}_\triangle$ and a simplicial function $F: D^{i+1}_\triangle \rightarrow X$ such that the $i$-sphere of $D^{i+1}_\triangle$ is $S^i_\triangle$ and $\left. F \right\vert_{S^i_\triangle} =f$. In that case, we will say that $F$ is an extension of $f$.
\end{definition}

\begin{remark}
We note that $X$ is $(-1)$-connected means that $X$ is non-empty and $X$ is $0$-connected means that it is non-empty and its $1$-skeleton is connected (as a graph).
\end{remark}

\begin{remark}
\label{Hurewicz thm rmk}
We note that by Hurewicz Theorem \cite[Theorem 7.5.5]{Spanier} a simplicial complex $X$ is $k$-connected if and only if it is simply connected and for every $0 \leq j \leq k$, $\widetilde{H}^j (X, \mathbb{Z}) =0$. Thus, if $X$ is $k$-connected, then by the universal coefficient Theorem for every $0 \leq j \leq k$, $\widetilde{H}^j (X) =0$. The converse is false: first, it could be that $\widetilde{H}^1 (X, \mathbb{Z}) =0$, but $X$ is not simply connected. Second, it can be that the homology with $\mathbb{F}_2$ coefficients vanish, but the homology with $\mathbb{Z}$ coefficients do not vanish: for instance for every $p >2$ prime, the lens space $L(p,1)$ is an orientable $3$-dimensional manifold such that $\widetilde{H}^j (L(p,1)) = 0$ for $j=0,1,2$, but $H^1 (L(p,1), \mathbb{F}_p) = \mathbb{F}_p$ (see more details in \cite[Examples 10.7, 13.6]{Bredon}). The lens space $L(p,1)$ is not a simplicial complex, but taking a triangulation of it yields a simplicial complex with the same homologies.
\end{remark}

For a simplicial complex $X$, if we assume that $X$ is $k$-connected, we define the $k$-th homotopy filling constant to be function as follows. First, for $0 \leq i \leq k$ and a simplicial map $f: S^i_\triangle \rightarrow X$, we define
$$\SFill_i (f) = \min \lbrace \vert D^{i+1}_\triangle (i+1) \vert : F: D^{i+1}_\triangle \rightarrow X \text{ is an extension of } f \rbrace.$$
A simplicial map $F: D^{i+1}_\triangle \rightarrow X$ extending $f$ such that $\vert D^{i+1}_\triangle (i+1) \vert = \SFill_i (f)$ will be called \textit{a minimal extension} of $f$.
Second, for $N \in \mathbb{N}$, we define
$$\SFill_i (N) = \max \lbrace \SFill_i (f) : f: S^i_\triangle \rightarrow X \text{ simplicial and } \vert  S^i_\triangle (i) \vert \leq N \rbrace.$$

With these definitions, we have analogues results to Proposition \ref{filling constants prop} and Theorem \ref{filling constants criterion for coboundary exp thm1}, namely:
\begin{theorem}
\label{simplicial filling constants thm}
Let $X$ be a pure finite $n$-dimensional simplicial complex. Let $0 \leq k \leq n-1$ and assume that $X$ is $k$-connected.
Define the following sequence of constants recursively: $N_{-1} = 1$ and for every $0 \leq j \leq k$,
$$N_j  = \SFill_{j} ((j+1) N_{j-1}  +1).$$
Then $\Crad_k (X) \leq N_k$.
\end{theorem}

The proof of Theorem \ref{simplicial filling constants thm} is similar to the proof of Proposition \ref{filling constants prop}, with the adaptation the we use admissible functions to construct the cone function:
\begin{proof}
Assume $X$ is $k$-connected and let $N_0,...,N_k$ be the constants of Theorem \ref{simplicial filling constants thm}.

Fix $\lbrace v \rbrace \in X(0)$. We will construct a simplicial complexes $Y^{-1},...,Y^k$ that will allow us to define a cone function with apex $v$. These complexes will be constructed iteratively (at each step of the construction we will add simplices to the previous complex).

\textbf{The complex $Y^{-1}$:} $Y^{-1}$ is a simplicial complex with a single vertex $Y^{-1} (0) = \lbrace \lbrace y \rbrace \rbrace$. We denote $D^\emptyset = Y^{-1}$ (note that $D^\emptyset$ is by definition a $0$-disc) and a function $F^\emptyset : D^0_\triangle \rightarrow X$ defined as $F^\emptyset (\lbrace y \rbrace) = \lbrace v \rbrace$ (where $v$ is our fixed vertex in $X$).

\textbf{The complex $Y^0$:} Define a complex $(Y^{-1})'$ as the $0$-dimensional complex defined as $(Y^{-1})' = Y^{-1} \cup \lbrace \lbrace x_u \rbrace : \lbrace u \rbrace \in X(0) \rbrace$. In other words, $(Y^{-1})'$ is isomorphic to the union of $Y^{-1}$ and the $0$-skeleton of $X(0)$. For every $\lbrace u \rbrace \in X(0)$, we define $S^{\lbrace u \rbrace} \subseteq (Y^{-1})'$ to be a triangulated $0$-sphere defined as $S^{\lbrace u \rbrace} = \lbrace \lbrace y \rbrace, \lbrace x_u \rbrace \rbrace$ and define a simplicial map $f^{\lbrace u \rbrace} : S^{\lbrace u \rbrace} \rightarrow X$ as
$$f^{\lbrace u \rbrace} (\lbrace y \rbrace) = F^\emptyset (D^\emptyset) = \lbrace v \rbrace, f^{\lbrace u \rbrace} (\lbrace x_u \rbrace) = \lbrace u \rbrace.$$
By our assumption, $X$ is $0$-connected and thus there is a minimal extension of $f^{\lbrace u \rbrace}$ which is a triangulated $1$-disc denoted $D^{\lbrace u \rbrace}$ and a simplicial function $F^{\lbrace u \rbrace} : D^{\lbrace u \rbrace} \rightarrow X$ extending $f^{\lbrace u \rbrace}$ such that $\vert D^{\lbrace u \rbrace} (1) \vert$ is minimal. Define $Y^0 = (Y^{-1})' \cup \bigcup_{\lbrace u \rbrace \in X(0)} D^{\lbrace u \rbrace}$. We note the following properties (some of them are just rephrasing of properties stated above in a convenient form):
\begin{enumerate}
\item $\lbrace \lbrace x_u \rbrace : \lbrace u \rbrace \in X(0) \rbrace \subseteq Y^0 (0)$.
\item For every $\lbrace u \rbrace \in X(0)$, $D^{\lbrace u \rbrace}$ is a subcomplex of $Y^0$ that is a triangulated $1$-disc with the sphere
\begin{dmath*}
S^{\lbrace u \rbrace} = \lbrace \lbrace y \rbrace, \lbrace x_u \rbrace \rbrace = \lbrace \lbrace x_u \rbrace \rbrace \cup D^\emptyset =  {\lbrace \lbrace x_u \rbrace : u \in \lbrace u \rbrace \rbrace \cup \bigcup_{\tau ' \subseteq \lbrace u \rbrace, \vert \tau ' \vert = \vert \lbrace u \rbrace \vert -1} D^{\tau '}}.
\end{dmath*}
\item The map $F^{\lbrace u \rbrace} : D^{\lbrace u \rbrace} \rightarrow X$ is a minimal extension of the map $f^{\lbrace u \rbrace} : S^{\lbrace u \rbrace} \rightarrow X$ that is defined
$$f^{\lbrace u \rbrace} (\tau ') = \begin{cases}
\lbrace u \rbrace & \tau ' = \lbrace x_u \rbrace \\
F^{\tau '} (D^{\tau '}) & \tau ' \in \partial_0 \lbrace u \rbrace
\end{cases}.$$
\item By the definition of $N_0$, for every $\lbrace u \rbrace \in X(0)$, $\vert D^{\lbrace u \rbrace} (1) \vert \leq N_0$.
\end{enumerate}

\textbf{The complex $Y^i$:} Assume by induction that $Y^{i-1}$ is defined and has the following properties:
\begin{enumerate}
\item For every $0 \leq j \leq i-1$ and every $\eta \in X(j)$, $\lbrace x_u \rbrace_{u \in \eta} \in Y^{i-1} (j)$.
\item For every $0 \leq j \leq i-1$ and every $\eta \in X(j)$, $D^{\eta}$ is a subcomplex of $Y^{i-1}$ that is a triangulated $j$-disc with the sphere
\begin{dmath*}
S^{\eta} = \lbrace \lbrace x_u \rbrace_{u \in \eta} \rbrace \cup \bigcup_{\eta ' \subseteq \eta, \vert \eta ' \vert = \vert \eta \vert -1} D^{\eta '}.
\end{dmath*}
\item For every $\eta \in X(i-1)$, the map $F^{\eta} : D^{\eta} \rightarrow X$ is a minimal extension of the map $f^{\eta} : S^{\eta} \rightarrow X$ that is defined
$$f^{\eta} (\eta ') = \begin{cases}
\eta & \eta ' = \lbrace x_u \rbrace_{u \in \eta} \\
F^{\eta '} (D^{\eta '}) & \eta' \in \partial_{i-1} \eta
\end{cases}.$$
\item For every $\eta \in X(i-1)$, $\vert D^{\eta} (i) \vert \leq N_{i-1}$.
\end{enumerate}

With these assumptions we construct $Y^i$ as follows: first, we define $(Y^{i-1})' = Y^{i-1} \cup \lbrace \lbrace x_u \rbrace_{u \in \tau} : \tau \in X(i) \rbrace$. For every $\tau \in X(i)$, we define a subcomplex of $(Y^{i-1})'$:
$$S^\tau = \lbrace x_u \rbrace_{u \in \tau} \cup \bigcup_{\tau ' \in \partial_i \tau} D^{\tau '}.$$
We note that this is glueing $i+1$ triangulated $i$-discs along their boundaries and that   the resulting $S^\tau$ is a triangulated $i$-sphere. We define a simplicial map $f^\tau : S^\tau \rightarrow X$ as
$$f^{\tau} (\tau ') = \begin{cases}
\tau & \tau ' = \lbrace x_u \rbrace_{u \in \tau} \\
F^{\tau '} (D^{\tau '}) & \tau ' \in \partial_i \tau
\end{cases}.$$
By the assumption that $X$ is $k$-connected, there is a triangulated $(i+1)$-disc that we will denote by $D^\tau$ and a minimal extension of $f^\tau$, which we will denote by $F^\tau : D^\tau \rightarrow X$. We note that properties (1)-(3) hold for $Y^i$. Also, we note that
$$\vert S^\tau (i) \vert = 1+ \sum_{\tau ' \in \partial_i \tau} \vert D^{\tau ' } (i) \vert \leq 1 + (i+1) N_{i-1}.$$
Thus, by definition, $\vert D^\tau (i+1) \vert \leq N_{i}$.

After this construction, we can define a cone function $\Cone_{k}^v: \bigoplus_{j=-1}^k C_{j} (X) \rightarrow \bigoplus_{j=-1}^k C_{j+1} (X)$ as follows: for every $-1 \leq i \leq k$ and every $\tau \in X(i)$, define
$$\Cone_{k}^v (\tau) = \sum_{\eta \in F^\tau (D^\tau) (i+1)} \eta.$$
Since the function $F^\tau$ is simplicial, it follows that
\begin{dmath*}
\partial_{i+1} \Cone_{k}^v (\tau) =
\partial_{i+1} \left( \sum_{\eta \in F^\tau (D^\tau) (i+1)} \eta \right) =
\sum_{\tau ' \in \partial_{i+1} F^\tau (D^\tau) (i+1)} \tau '  =
\sum_{\tau ' \in F^\tau (\partial_{i+1} D^\tau) (i+1)} \tau '  =
\sum_{\tau ' \in F^\tau (S^\tau) (i+1)} \tau '  =
\sum_{\tau ' \in f^\tau (S^\tau) (i+1)} \tau '  =
\tau + \sum_{\tau ' \in \partial_i \tau} \sum_{\eta \in F^{\tau '} (D^{\tau '}) (i)} \eta =
\tau + \Cone_{k}^v (\partial_i \tau).
\end{dmath*}
Thus, $\Cone_{k}^v$ is indeed a cone function and it follows that for every $\tau \in X(k)$, $\vert \Cone_{k}^v (\tau) \vert \leq \vert F^\tau (D^\tau) (i+1) \vert \leq N_k$, thus $\Crad_k (X) \leq N_k$ as needed.

\end{proof}

Combining Theorem \ref{simplicial filling constants thm} with Theorem \ref{Crad thm}, we deduce the following:

\begin{theorem}
\label{filling constants criterion for coboundary exp thm2}
Let $X$ be a pure finite $n$-dimensional simplicial complex. Assume that $X$ is $k$-connected and let $N_k$ be the constant defined in Theorem \ref{simplicial filling constants thm} above. Assume that there is a group $G$ of automorphisms of $X$ acting transitively on $X(n)$. Then $\Exp^k_b (X) \geq \frac{1}{{n+1 \choose k+1} N_k}$.
\end{theorem}

\section{Symmetric complexes and coset complexes}
\label{Symmetric complexes and coset complexes sec}
In Theorem \ref{Crad thm} above, we saw how coboundary expansion of a simplicial complex $X$ can be deduced from the cone radius under the assumption of strong symmetry, i.e., under the assumption that there is a group $G$ acting transitively on the top dimensional simplices of $X$. Below, we note that this assumption, together with the assumption that $X$ is a partite clique complex actually imply that $X$ can be identified with cosets of $G$. This result was already known - for instance it is stated in \cite{DiaGeomBook} in the language of coset geometries (for a dictionary between coset geometries and simplicial complexes see \cite{KOCosetGeom}) and we include the proof below for completeness.

We start with the following definitions:
\begin{definition}[Clique complex]
A simplicial complex $X$ is called a clique complex if every clique in its one-skeleton spans a simplex, i.e., for every $\lbrace v_0 \rbrace,...,\lbrace v_k \rbrace \in X(0)$ if $\lbrace v_i, v_j \rbrace \in X(1)$ for every $0 \leq i < j \leq k$, then $\lbrace v_0,...,v_k \rbrace \in X(k)$.
\end{definition}

\begin{definition}[Partite complex, type]
Let $X$ be a pure $n$-dimensional simplicial complex over a vertex set $V$. The complex $X$ is called \textit{$(n+1)$-partite}, if there are disjoint sets $S_0,...,S_n \subset V$, called \textit{the sides of $X$}, such that for  $V = \bigcup S_i$ and for every $\sigma \in X(n)$ and every $0 \leq i \leq n$, $\vert \sigma \cap S_i \vert = 1$, i.e., every $n$-dimensional simplex has exactly one vertex in each of the sides of $X$.  In a pure $n$-dimensional, $(n+1)$-partite complex $X$, each simplex $\sigma \in X(k)$ has a \textit{type} which is a subset of $\lbrace 0,..., n\rbrace$ of cardinality $k+1$ that is defined by $\type (\sigma) = \lbrace i : \sigma \cap S_i \neq \emptyset \rbrace$.
\end{definition}

\begin{definition}[Coset complex]
Given a group $G$ with subgroups $K_{\lbrace i \rbrace}, i \in \I$, where $\I$ is a finite set. The \textit{coset complex} $X=X(G, (K_{\lbrace i \rbrace})_{i \in \I})$ is a simplicial complex defined as follows:
\begin{enumerate}
\item The vertex set of $X$ is composed of disjoint sets $S_i = \lbrace g K_{\lbrace i \rbrace} : g \in G \rbrace$.
\item For two vertices $g K_{\lbrace i \rbrace}, g' K_{\lbrace j \rbrace}$ where $i,j \in \I, g,g' \in G$, $\lbrace g K_{\lbrace i \rbrace}, g' K_{\lbrace j \rbrace} \rbrace \in X(1)$ if $i \neq j$ and  $g K_{\lbrace i \rbrace} \cap g' K_{\lbrace j \rbrace} \neq \emptyset$.
\item The simplicial complex $X$ is the clique complex spanned by the $1$-skeleton defined above, i.e., $\lbrace g_0 K_{\lbrace i_0 \rbrace},..., g_k  K_{\lbrace i_k \rbrace} \rbrace \in X(k)$ if for every $0 \leq j,j' \leq k$, $g_j K_{\lbrace i_j \rbrace} \cap g_{j'} K_{\lbrace i_{j'} \rbrace} \neq \emptyset$.
\end{enumerate}
\end{definition}

\begin{observation}
For $G$ and $K_{\lbrace i \rbrace}, i \in \I$ as above, the coset complex $X=X(G, (K_{\lbrace i \rbrace})_{i \in \I})$ is a clique complex and $G$ acts on $X$ simplicially by $g' . g K_{\lbrace i \rbrace} = g'g K_{\lbrace i \rbrace}, \forall g,g' \in G, i \in I$.
\end{observation}

The following proposition shows that under some assumptions, a strongly symmetric complex is always (isomorphic to) a coset complex:

\begin{proposition}
\label{symmetric complexes and coset complexes prop}
Let $Y$ be a pure $n$-dimensional clique complex and let $G$ be a group acting by type preserving automorphisms on $Y$ such that the action is transitive on $n$-simplices. Fix $\lbrace v_0,...,v_n \rbrace \in Y (n)$ such that $v_i$ is of type $i$ and denote $K_{\lbrace i \rbrace} = \Stab (v_i)$. Then $Y$ is isomorphic to $X=X(G, (K_{\lbrace i \rbrace})_{i \in \lbrace 0,...,n \rbrace})$. Furthermore, there is an isomorphism $\Phi : X \rightarrow Y$ that is equivariant with respect to the action of $G$.
\end{proposition}

\begin{proof}
Define $\Phi : X(0) \rightarrow Y(0)$ by $\Phi (g K_{\lbrace i \rbrace}) = g. v_i$.

The map $\Phi$ is well-defined: if $g K_{\lbrace i \rbrace} = g' K_{\lbrace i \rbrace}$, then there is $h \in K_{\lbrace i \rbrace}$ such that $g' = gh$ and therefore $g'.v_i = gh.v_i = g.v_i$.

The map $\Phi$ is injective: $\Phi (g K_{\lbrace i \rbrace}) = \Phi (g' K_{\lbrace i \rbrace})$ implies that $g.v_i = g'.v_{i}$ or equivalently $g^{-1} g' \in K_{\lbrace i \rbrace})$. Thus,
$$g K_{\lbrace i \rbrace} = g g^{-1} g' K_{\lbrace i \rbrace} = g' K_{\lbrace i \rbrace}.$$

The map $\Phi$ is surjective: the group $G$ acts transitively on $n$ dimensional simplices in $Y$ and is type preserving and thus acts transitively on vertices of the same type. As a result if $v$ is a vertex in $Y$ of type $i$, there is $g \in G$ such that $g.v_i = v$ and therefore $\Phi (g K_{\lbrace i \rbrace}) = v$.

Observe that $\Phi$ is also equivariant under the action of $G$.

When extended to a map between $X$ and $Y$, the map $\Phi$ is a simplicial isomorphism: we note that since both $X$ and $Y$ are clique complexes, it is enough to check that $\Phi$ maps an edge in $X$ to an edge in $Y$, and, vice-versa, the preimage of an edge in $Y$ is an edge in $X$. Let $g K_{\lbrace i \rbrace}, g' K_{\lbrace i' \rbrace}$ be vertices in $X$ that are connected by an edge. Then by definition, there is $g'' \in G$ such that $g'' K_{\lbrace i \rbrace} = g K_{\lbrace i \rbrace}$ and $g'' K_{\lbrace i' \rbrace} = g' K_{\lbrace i' \rbrace}$. Thus, $\Phi (g K_{\lbrace i \rbrace}) = g''.v_i$ and $\Phi (g' K_{\lbrace i' \rbrace}) = g''.v_{i'}$. The vertices $v_{i}, v_{i'}$ are connected by an edge in $Y$ and $G$ acts simplicially on $Y$ and therefore $g''.v_i$ and $g''. v_{i'}$ are also connected by an edge as needed.

In the other direction, let $v,u$ vertices in $Y$ such that $\lbrace v,u \rbrace \in Y(1)$. The complex $Y$ is $(n+1)$-partite and therefore there are $0 \leq i,i' \leq n, i \neq i'$ such that $v$ is of type $i$ and $u$ is of type $i'$. Without loss of generality, we will assume that $v$ is of type $0$ and $u$ is of type $1$. The complex $Y$ is pure $n$-dimensional and therefore there are vertices $u_2,...,u_n$ such that $\lbrace v,u, u_2,...,u_n \rbrace \in Y(n)$. The group $G$ acts transitively on $n$-dimensional simplices of $Y$ and as a result there is $g \in G$ such that $g. \lbrace v_0,...,v_n \rbrace = \lbrace v,u, u_2,...,u_n \rbrace$ and the action is type preserving which implies that $g. v_0 = v, g.v_1 = u$. Thus for that $g$, $\Phi (g K_{\lbrace 0 \rbrace})=v$ and $\Phi (g K_{\lbrace 1 \rbrace})=u$, i.e, the preimage of $\lbrace v,u \rbrace$ is $\lbrace g K_{\lbrace 0 \rbrace}, g K_{\lbrace 1 \rbrace} \rbrace \in X(1)$ as needed.
\end{proof}

Thus, under the assumption of partiteness, instead of working with symmetric simplicial complexes, we can work with coset complexes. However, one should note the following issue: the action of $G$ on the coset complex $X= X(G, (K_{\lbrace i \rbrace})_{i \in \lbrace 0,...,n \rbrace})$ is always transitive on vertices of the same type in $X$, but, in general, it need not be transitive on $X(n)$. The following result gives a criterion for transitivity on $X(n)$:

\begin{theorem}\cite[Theorem 1.8.10]{DiaGeomBook}
\label{transitive action thm}
Let $G$ be a group with subgroups $K_{\lbrace i \rbrace}, i \in \I$, where $\I$ is a finite set. Denote $X = X(G, (K_{\lbrace i \rbrace})_{i \in \I})$ to be the coset complex defined above. Denote further for every $\emptyset \neq \tau \subseteq I$, $K_{\tau} = \bigcap_{i \in \tau} K_{\lbrace i \rbrace}$ and $K_\emptyset = G$. The action of $G$ on $X(n)$ is transitive (and thus $X$ is strongly symmetric) if and only if for every $\tau \subsetneqq I$ and every $i \in I \setminus \tau$, $K_{\tau} K_{\lbrace i \rbrace}  = \bigcap_{j \in \tau} K_{\lbrace j \rbrace} K_{\lbrace i \rbrace}$.
\end{theorem}

\section{Bounding the first two filling constants for coset complexes}
\label{Bounding the first two filling constants for coset complexes sec}

In Theorem \ref{filling constants criterion for coboundary exp thm2}, we showed that for a symmetric simplicial complex, the high order Cheeger constants can be bounded using the filling constants of the complex. In the previous section we have seen that under suitable assumptions coset complexes are examples of symmetric complexes (see Theorem \ref{transitive action thm}) and that under the assumption of partiteness, a strongly symmetric clique complex is a coset complex (see Proposition \ref{symmetric complexes and coset complexes prop}). Thus, under suitable assumptions, the problem of bounding the Cheeger constants of a complex is reduced to bounding the filling constants of the correspondent coset complex.

Throughout this section, $G$ is a finite group with subgroups $K_{\lbrace i \rbrace}, i \in \I$, where $\I$ is a finite set and $\vert I \vert \geq 2$, and we denote $X = X(G, (K_{\lbrace i \rbrace})_{i \in \I})$ to be the coset complex defined above. We always assume that $X$ is strongly symmetric.

Below, we will show how to bound the first two filling constants of a coset complex, based to the properties of the symmetry group $G$ and the stabilizer subgroups $K_{\lbrace i \rbrace}, i \in \I$.

There are known criteria in the literature for connectedness and simply connectedness of coset complexes (see for instance \cite{AbelsH} or \cite[Chapter 6]{Garst} and reference therein). Namely,  Abels and Holtz \cite{AbelsH} proved the following:
\begin{theorem}\cite[Theorem 2.4]{AbelsH}
Let $G, K_{\lbrace i \rbrace}, i \in \I$ and $X$ be as above.
\begin{enumerate}
\item The $1$-skeleton of $X$ is connected if and only if the subgroups $K_{\lbrace i \rbrace}, i \in \I$ generate $G$.
\item The simplicial complex $X$ is connected and simply connected if and only if there is a presentation of $G$ of the form $G = \langle \bigcup_{i} K_{\lbrace i \rbrace} \vert R \rangle$, where every relation in $R$ is a relation in some $K_{\lbrace i \rbrace}$, i.e., every relation is of the form $g_1 ... g_k = e$ where there is some $i \in I$ such that $g_1,...,g_k \in K_{\lbrace i \rbrace}$.
\end{enumerate}
\end{theorem}

Below, we prove quantitative versions of these criteria in order to bound $N_0$ and $N_1$.

\subsection{Bound on $N_0$}

Recall that for a group $G$ with subgroups $K_{\lbrace i \rbrace}, i \in \I$, we say that $K_{\lbrace i \rbrace}, i \in \I$ \textit{boundedly generate} $G$ if there is some $L \in \mathbb{N}$ such that every $g \in G$ can be written as a product of at most $L$ elements in $\bigcup_{i \in I} K_{\lbrace i \rbrace}$. We will show below that in this case $L+1$ is a bound on $N_0$.

\begin{lemma}
\label{condition of edge connecting lemma}
Two vertices $g K_{\lbrace i \rbrace}$, $g' K_{\lbrace i' \rbrace}$ in $X(0)$ are connected by an edge if and only if there is $h ' \in K_{\lbrace i ' \rbrace}$ such that $g K_{\lbrace i \rbrace} = g' h' K_{\lbrace i \rbrace}$.
\end{lemma}

\begin{proof}
Recall that by definition  $g K_{\lbrace i \rbrace}$, $g' K_{\lbrace i' \rbrace}$ are connected by an edge if and only if $ g K_{\lbrace i \rbrace} \cap  g'  K_{\lbrace i ' \rbrace} \neq \emptyset$.

Assume that there is $h' \in K_{\lbrace i ' \rbrace}$ such that $g K_{\lbrace i \rbrace} = g' h' K_{\lbrace i \rbrace}$. Then
$$g K_{\lbrace i \rbrace} \cap  g'  K_{\lbrace i ' \rbrace} = g' h' K_{\lbrace i ' \rbrace} \cap g' h' K_{\lbrace i ' \rbrace}.$$
Thus, $g'h' \in  g K_{\lbrace i \rbrace} \cap  g'  K_{\lbrace i ' \rbrace}$ and the intersection is not empty.

Conversely, if $g K_{\lbrace i \rbrace} \cap g' K_{\lbrace i ' \rbrace} \neq \emptyset$, then there are $h \in K_{\lbrace i \rbrace}, h' \in K_{\lbrace i ' \rbrace}$ such that
$gh = g' h'$ and it follows that $g = g' h' h^{-1}$. Thus,
$$g K_{\lbrace i \rbrace} = g' h' h^{-1} K_{\lbrace i \rbrace} = g' h'  K_{\lbrace i \rbrace},$$
as needed.
\end{proof}

\begin{proposition}
Let $G$ be a group  with subgroups $K_{\lbrace i \rbrace}, i \in \I$, where $\I$ is a finite set and $\vert I \vert \geq 2$. Denote $X = X(G, (K_{\lbrace i \rbrace})_{i \in \I})$ to be the coset complex defined above. Assume that subgroups $K_{\lbrace i \rbrace}, i \in \I$ boundedly generate $G$, and denote
$$L = \max_{g \in G} \min \left\lbrace l : g = g_1...g_l \text{ and } g_1,...,g_l \in \bigcup_i K_{\lbrace i \rbrace} \right\rbrace.$$
Then the diameter of the $1$-skeleton of $X$ is bounded by $L+1$.
\end{proposition}

\begin{proof}
Assume that $K_{\lbrace i \rbrace}, i \in \I$ boundedly generate $G$. Recall that the group $G$ acts simplicially on $X$ and transitively on vertices of the same type. Thus, it is sufficient to prove that for every $i',i'' \in I$ and every $g \in G$, the vertices $K_{\lbrace i' \rbrace}$ and $g K_{\lbrace i'' \rbrace}$ are connected by a path of length less or equal to $L+1$.

By our assumption, there is $l \leq L$ and $g_1,...,g_l \in \bigcup_i K_{\lbrace i \rbrace}$ such that $g = g_1...g_l$. We will show that there is a path of length $\leq l+1$ connecting  $K_{\lbrace i' \rbrace}$ and $g K_{\lbrace i'' \rbrace}$.

Let $i_j \in I$ such that $g_{i_j} \in K_{\lbrace i \rbrace}$ for every $1 \leq j \leq l-1$. Then by Lemma \ref{condition of edge connecting lemma}, for every $1 \leq j \leq l$, $g_1 ... g_{j+1} K_{\lbrace i_{j+1} \rbrace}$ and  $g_1 ... g_{j} K_{\lbrace i_{j} \rbrace}$ are connected by an edge. Thus there is a path of length $l-1$ in the $1$-skeleton of $X$ connecting $g_1 K_{\lbrace i_1 \rbrace}$ and $g_1 ... g_l K_{\lbrace i_l \rbrace} = g K_{\lbrace i_l \rbrace}$. Note that $g_1 K_{\lbrace i_1 \rbrace} = K_{\lbrace i_1 \rbrace}$ and thus $g_1 K_{\lbrace i_1 \rbrace} \cap K_{\lbrace i' \rbrace}  \neq \emptyset$. Also, $g K_{\lbrace i_l \rbrace} \cap g K_{\lbrace i'' \rbrace} \neq \emptyset$ and thus there is a path of length $l+1$ connecting $K_{\lbrace i' \rbrace}$ and $g K_{\lbrace i'' \rbrace}$.
\end{proof}

\begin{corollary}
\label{bound on N_0 coro}
Let $G$ be a group with subgroups $K_{\lbrace i \rbrace}, i \in \I$, where $\I$ is a finite set and $\vert I \vert \geq 2$. Denote $X = X(G, (K_{\lbrace i \rbrace})_{i \in \I})$ and let $N_0$ be the constant defined in Theorem \ref{simplicial filling constants thm}. If $K_{\lbrace i \rbrace}, i \in \I$ boundedly generate $G$, then $N_0$  is bounded by
$$1+\max_{g \in G} \min \left\lbrace l : g = g_1...g_l \text{ and } g_1,...,g_l \in \bigcup_i K_{\lbrace i \rbrace} \right\rbrace.$$
\end{corollary}

\begin{proof}
Note that by definition $N_0$ is the radius of the $1$-skeleton of $X$ and thus it is bounded by the diameter of the $1$-skeleton of $X$ and the corollary follows.
\end{proof}

\subsection{Bound on $N_1$}

In order to state the criterion, we recall some definitions regarding the Dehn function of a group presentation.

Let $G = \langle S \vert R \rangle$ be a finitely presented group, where $S$ is a symmetric generating set of $G$ and $R$ is a set of relations. Denote $\Free (S)$ to be the free group with a generating set $S$. Without loss of generality, we can assume that $R \subseteq \Free (S)$ are cyclically reduced words.  A \textit{relation} in $G$ is a freely reduced word $w \in \Free (S)$ such that $w=1$ in $G$. We note that every relation is in the normal closure of $R$ in $\Free (S)$.

Given a relation $w$, the \textit{area} of $w$, denoted $\Area (w)$, is the minimal number $a \in \mathbb{N} \cup \lbrace 0 \rbrace$ such there are $r_1,...,r_a \in R \cup R^{-1}$ and $u_1,...,u_a \in \Free (S)$ such that
$$w = u_1^{-1} r_1 u_1 ... u_a^{-1} r_a u_a \text{ in } \Free (S).$$

\begin{definition}[The Dehn function]
\label{Dehn func def}
The \textit{Dehn function of $G$ with respect to $S,R$} is the function $\Dehn : \mathbb{N} \rightarrow \mathbb{N}$ defined as
$$\Dehn (m) = \max \lbrace \Area (w) :w = 1 \text{ in } G, \vert w \vert \leq m \rbrace,$$
where $\vert w \vert$ denotes the word length of $w$ in $\Free (S)$.
\end{definition}

The idea behind the Dehn function is that it counts how many reduction moves are needed to reduce $w$ to the trivial word. More precisely, following \cite[Chapter 8]{OfficeHBook}, we define the following moves on a word $w$:
\begin{enumerate}
\item Free reduction: remove a sub-word of the form $s s^{-1}$ or $s^{-1} s$ within the word $w$, where $s \in S$.
\item Applying a relation: replace a subword $w'$ in $w$ with a new subword $w''$ where $w' (w'')^{-1}$ or $w'' (w')^{-1}$ is a cyclic permutation of a word in $R \cup R^{-1}$.
\end{enumerate}

The only fact that we will need regarding the Dehn function is that is one can reduce a word $w=1$ with $\vert w \vert =m$ to the trivial word using at most $\Dehn (m)$ applications of relations and at most $\left(\max_{r \in R} \vert r \vert \right) \Dehn (m) + m$ free reductions (see \cite[Chapter 8]{OfficeHBook}). For a far more extensive introduction regarding Dehn functions, the reader is referred to Riley's or Bridson's expository articles on this subject (see \cite[Chapter 8]{OfficeHBook} or \cite{Bridson}).

\begin{theorem}
\label{bound on N1 thm}
Let $G$ be a group with subgroups $K_{\lbrace i \rbrace}, i \in \I$, where $\I$ is a finite set and $\vert I \vert \geq 3$ and assume that the subgroups $K_{\lbrace i \rbrace}, i \in \I$ are all finite and that they generate $G$. Denote $X = X(G, (K_{\lbrace i \rbrace})_{i \in \I})$ to be the coset complex defined above. For every $i \in I$, denote $R_{i}$ to be all the non-trivial relations in the multiplication table of $K_{\lbrace i \rbrace}$, i.e., all the relations of the form $g_1 g_2 g_3 =e$, where $g_1,g_2, g_3 \in K_{\lbrace i \rbrace} \setminus \lbrace e \rbrace$. Assume that $G = \langle \bigcup_{i} K_{\lbrace i \rbrace} \vert \bigcup_{i} R_i \rangle$ and let $\Dehn$ denote the Dehn function of $G$ with respect to this presentation of $G$. Then for triangulated $1$-sphere $S^1_\triangle$ such that $\vert S^1_\triangle (1) \vert \leq m$ and every simplicial map $f: S^1_\triangle \rightarrow X_1$, there is a an extension $F:D^1_\triangle \rightarrow X$ with
$$\vert D^1_\triangle (2) \vert \leq p (m, \Dehn (m)),$$
where $p(x,y)$ is the polynomial
$$p(x,y) = 16 y^2 + 6 x^2 +20 xy + 24 y+ 10x +1.$$
\end{theorem}


Before proving this Theorem, we will need to set up some terminology, notation and lemmata.


We recall the following definitions taken from \cite{TopAndGeomBook}:
\begin{definition}
Two simplicial maps $f, f' : S^1_\triangle \rightarrow X$ are called contiguous if for every simple $\sigma \in S^1_\triangle$, $f (\sigma) \cup f' (\sigma)$ is a simplex in $X$.
Also, $f, f' : S^1_\triangle \rightarrow X$ are called contiguously equivalent, if there are simplicial maps $f_0,...,f_k :  S^1_\triangle \rightarrow X$ where $f = f_0, f' =f_k$ and $f_{i}, f_{i+1}$ are contiguous for all $0 \leq i \leq k$.
\end{definition}

A basic fact is that if $f,f' : S^1_\triangle \rightarrow X$ are contiguous, then they are homotopic and thus $f$ and $f'$ are homotopy equivalent. Furthermore, $f$ can be extended to a map $F: D^2_\triangle \rightarrow X$ if and only if $f'$ can be extended to a map $F': (D^2_\triangle)' \rightarrow X$. We want to quantify this statement:
\begin{lemma}
\label{contiguous maps lemma}
Let $f,f' :  S^1_\triangle \rightarrow X$ be contiguous simplicial maps. Assume that $F: D^2_\triangle \rightarrow X$ is a minimal extension of $f$, then there is a minimal extension of $f'$, $F': (D^2_\triangle)' \rightarrow X$ such that
$$\vert (D^2_\triangle)' (2) \vert \leq \vert D^2_\triangle (2) \vert + 2 \vert \lbrace \lbrace u \rbrace \in S^1_\triangle (0) : f (\lbrace u \rbrace) \neq f' (\lbrace u \rbrace) \rbrace \vert.$$
\end{lemma}

\begin{proof}
By induction on $\vert \lbrace \lbrace u \rbrace \in S^1_\triangle (0) : f (\lbrace u \rbrace) \neq f' (\lbrace u \rbrace) \rbrace \vert$ it is enough to prove that if $f$, $f'$ are contiguous and there is only a single $\lbrace u \rbrace \in S^1_\triangle (0)$ such that $f (\lbrace u \rbrace) \neq f' (\lbrace u \rbrace)$, it follows that
$$\vert (D^2_\triangle)' (2) \vert \leq \vert D^2_\triangle (2) \vert + 2.$$
Let $F : D^2_\triangle \rightarrow X$ be a minimal extension of $f$ and denote by $\lbrace v \rbrace, \lbrace w \rbrace \in  S^1_\triangle (0)$ the two neighbours of $\lbrace u \rbrace$. We will show that there is an extension of $f'$, $F': (D^2_\triangle)' \rightarrow X$ such that
$$\vert (D^2_\triangle)' (2) \vert = \vert D^2_\triangle (2) \vert +2.$$
We do not claim that the extension $F'$ which we will define below is minimal and thus for a minimal extension of $f'$ there is an inequality.

Define $(D^2_\triangle)'$ as follows: add a vertex $\lbrace u' \rbrace$ to $D^2_\triangle$, connect $\lbrace u' \rbrace$ to the vertices $\lbrace u\rbrace, \lbrace v \rbrace, \lbrace w \rbrace \in S^1_\triangle \subseteq D^2_\triangle$ and take $(D^2_\triangle)'$ to be the resulting clique complex. In other words,
$$(D^2_\triangle)' (0) = D^2_\triangle (0) \cup \lbrace \lbrace u' \rbrace \rbrace,$$
$$(D^2_\triangle)' (1) = D^2_\triangle (1) \cup \lbrace \lbrace u', u \rbrace,\lbrace u', v \rbrace, \lbrace u', w \rbrace  \rbrace,$$
$$(D^2_\triangle)' (2) = D^2_\triangle (2) \cup \lbrace \lbrace u', u, v \rbrace,\lbrace u', u, w \rbrace  \rbrace.$$
Define $F': (D^2_\triangle)' \rightarrow X$ to be the following map
$$F' (\sigma) = \begin{cases}
F (\sigma) & u' \notin \sigma \\
F (\sigma \setminus \lbrace u ' \rbrace) \cup f' (\lbrace u \rbrace) & u' \in \sigma
\end{cases}.$$
We note that $\left. F' \right\vert_{S^1_\triangle} = f'$. To verify that $F'$ is simplicial it is enough to verify that it is simplicial on $\lbrace \lbrace u', u, v \rbrace,\lbrace u', u, w \rbrace  \rbrace$. Indeed, since $f,f'$ are contiguous, if follows that $f (\lbrace u, v \rbrace) \cup f' (\lbrace u, v \rbrace) \in X$, and we note that
$$F' (\lbrace u', u, v \rbrace) = f (\lbrace u, v \rbrace) \cup f' (\lbrace u, v \rbrace) \in X.$$
Similarly, $F'$ is simplicial on $\lbrace u', u, w \rbrace$. Thus $F'$ is an extension of $f'$ and $\vert (D^2_\triangle)' (2) \vert = \vert D^2_\triangle (2) \vert +2$ as needed (note that $F'$ is not necessarily minimal and thus the inequality).
\end{proof}
This Lemma allows us to pass to contiguous maps and ``bookkeep'' the maximal change in the number of $2$-simplices in a minimal extension.

Next, we will use the fact that $X$ is a coset complex code to code the paths in $X$. For convenience we define $S^1_m$ to be the triangulated $1$-sphere with vertices $S^1_m (0) = \lbrace \lbrace j \rbrace : 0 \leq j \leq m-1 \rbrace$ and edges $S^1_m (1) = \lbrace \lbrace j,j+1 \rbrace : 0 \leq j \leq m-1 \rbrace$ where $j+1$ is taken modulo $m$. Using this notation, the vertices of $S^1_m$ are $\lbrace 0,...,m-1 \rbrace$ and a simplicial map $f: S^1_m \rightarrow X$ is determined by $f(j)$ where $0 \leq j \leq m-1$.

\begin{definition}
Let $G$, $K_{\lbrace i \rbrace}, i \in \I$ and $X$ be as above. For $(g_1,...,g_m) \in \left( \bigcup_{i} K_{\lbrace i \rbrace} \right)^m$ with $g_1 ... g_m =e$, a simplicial map $f: S^1_m \rightarrow X$ is called a $(g_1,...,g_m)$-map if there are $i_0,...,i_{m-1} \in \I$ such that:
\begin{enumerate}
\item $f (0) = K_{\lbrace i_0 \rbrace} = g_1 ... g_m K_{\lbrace i_0 \rbrace}$.
\item For every $1 \leq j \leq m-1$, $f(j) = g_1 ... g_j K_{\lbrace i_j \rbrace}$.
\item It holds that $g_1 ... g_{m-1} \in K_{\lbrace i_{m-1} \rbrace}$ or, in other words, $f(m-1) = K_{\lbrace i_{m-1} \rbrace}$.
\end{enumerate}
\end{definition}

\begin{lemma}
\label{translation to (g_1,...,g_m) lemma}
Let $G$, $K_{\lbrace i \rbrace}, i \in \I$ and $X$ be as above. For every simplicial map $f: S^1_m \rightarrow X$ there is $g \in G$ and $(g_1,...,g_m) \in \left( \bigcup_{i} K_{\lbrace i \rbrace} \right)^m$ with $g_1 ... g_m =e$ such that $g.f$ is a $(g_1,...,g_m)$-map.
\end{lemma}

\begin{proof}
Recall that $G$ acts transitively on the $n$-simplices of $X$, thus if $f(0) = h K_{\lbrace i_0 \rbrace}$ and $f(m-1) = h' K_{\lbrace i_{m-1} \rbrace}$ there is $g \in G$ such that $g. f(0) = K_{\lbrace i_0 \rbrace}$ and $g. f(m-1) = K_{\lbrace i_{m-1} \rbrace}$. By the definition of $g.f$ there are $h_1,...,h_{m-2} \in G$ and $i_1,...,i_{m-2} \in \I$ such that for every $1 \leq j \leq j-2$, $g. f(j) = h_j K_{\lbrace i_{j} \rbrace}$.

Denote $h_0=h_{m-1} =e$. By Lemma \ref{condition of edge connecting lemma} there are $g_1,...,g_{m-1} \in \bigcup_{i} K_{\lbrace i \rbrace}$ such that $h_{j} K_{\lbrace i_{j} \rbrace} =  h_{j-1} g_j K_{\lbrace i_{j} \rbrace}$ for every $0 \leq j \leq m-1$. Since $h_0 =e$, it follows that
$g. f (j) = g_1...g_j K_{\lbrace i_{j} \rbrace}$ for every $0 \leq j \leq m-1$. Note that $h_{m-1} =e$ implies that $g_1...g_{m-1} K_{\lbrace i_{m-1} \rbrace} = K_{\lbrace i_{m-1} \rbrace}$ and therefore $g_1...g_{m-1} \in  K_{\lbrace i_{m-1} \rbrace}$ and we can choose $g_{m} = (g_1...g_{m-1})^{-1} \in K_{\lbrace i_{m-1} \rbrace}$. Thus $g_1...g_{m} = e$ and $g_1...g_{m-1} K_{\lbrace i_{0} \rbrace} = K_{\lbrace i_{0} \rbrace}$ as needed.
\end{proof}

Since $G$ acts by automorphisms on $X$, the takeaway for this Lemma is that it is enough to prove Theorem \ref{bound on N1 thm} on $(g_1,...,g_m)$-maps.

\begin{lemma}
\label{relation moves lemma}
Let $f : S^1_m \rightarrow X$ be a $(g_1,...,g_m)$-map.
If for some $1 \leq j' \leq m-1$, there is $i' \in I$ such that $g_{j'}, g_{j'+1} \in K_{\lbrace i' \rbrace}$ and $g_{j'} g_{j'+1} = g_{j'+1}'$, then there is $f' : S^1_m \rightarrow X$ such that $f'$ is a $(g_1,...g_{j'-1},e,g_{j'+1}',...,g_m)$-map and $\SFill_1 (f) \leq \SFill_1 (f')+ 2$.
\end{lemma}

\begin{proof}
Assume $g_{j'} g_{j'+1} = g_{j'+1}'$ and $g_{j'}, g_{j'+1} \in K_{\lbrace i' \rbrace}$. Define $f' : S^1_m \rightarrow X$ to be the map induced by:
$$f' (j) = \begin{cases}
f(j) & j \neq j' \\
g_1 ... g_{j' -1} K_{\lbrace i' \rbrace} & j = j'
\end{cases}.$$
If we show that $f'$ is a simplicial map then it is a $(g_1,...g_{j'-1},e,g_{j'+1}',...,g_m)$-map. If we also show that $f'$ is contiguous to $f$, it will follow from Lemma \ref{contiguous maps lemma} that $\SFill_1 (f) \leq \SFill_1 (f')+ 2$.

In order to show that $f'$ is simplicial, we need to verify that $\lbrace f' (j'-1), f'(j') \rbrace \in X$ and $\lbrace f' (j'+1), f'(j') \rbrace \in X$. Note that
$$f' (j'-1) \cap f'(j') = g_1 ... g_{j' -1} K_{\lbrace i_{j'-1} \rbrace} \cap g_1 ... g_{j' -1} K_{\lbrace i' \rbrace},$$
thus $g_1 ... g_{j' -1} \in f' (j'-1) \cap f(j')$ and therefore $\lbrace f' (j'-1), f(j') \rbrace \in X$. Also, note that since $g_{j'} \in K_{\lbrace i' \rbrace}$ it follows that
$$f' (j') =  g_1 ... g_{j' -1} K_{\lbrace i' \rbrace} = g_1 ... g_{j'} K_{\lbrace i' \rbrace}.$$
Thus $g_1 ... g_{j'} \in f' (j'+1) \cap f'(j')$ and $\lbrace f' (j'+1), f' (j') \rbrace \in X$.

Next, we will show that $f$ and $f'$ are contiguous. First, we observe that
$$f (j') \cap f'(j') = g_1 ... g_{j'} K_{\lbrace i_j \rbrace} \cap g_1 ... g_{j'} K_{\lbrace i' \rbrace},$$
thus $\lbrace f (j'), f'(j') \rbrace \in X$. Second, we note that
$$\lbrace f (j'-1), f(j') \rbrace \cup \lbrace f' (j'-1), f'(j') \rbrace = \lbrace f (j'-1), f(j'), f' (j') \rbrace,$$
and we already know that $\lbrace f (j'-1), f(j') \rbrace \in X$ and $\lbrace f (j'-1), f'(j') \rbrace \in X$. Thus, if we show that $\lbrace f (j'), f'(j') \rbrace \in X$ it will follow that $\lbrace f (j'-1), f(j'), f' (j') \rbrace \in X$ (since $X$ is a clique complex). A similar argument shows that $\lbrace f (j'+1), f(j') \rbrace \cup \lbrace f' (j'+1), f'(j') \rbrace \in X$. Thus $f$ and $f'$ are contiguous as needed.

\end{proof}

An immediate corollaries of this Lemma are:

\begin{corollary}
\label{Identity ordering coro}
Let $f : S^1_m \rightarrow X$ be a $(g_1,...,g_m)$-map where  $(g_1,...,g_m) \in \left( \bigcup_{i} K_{\lbrace i \rbrace} \right)^m$ and $g_1 ... g_m =e$. If for some $1 \leq j' \leq m$, $g_{j'} =e$, then there is $f' : S^1_m \rightarrow X$ such that $f'$ is a $(e,g_1,...g_{j'-1},g_{j'+1},...,g_m)$-map and $\SFill_1 (f) \leq \SFill_1 (f')+ 2(j'-1)$.
\end{corollary}

\begin{proof}
By induction on $j'$, it is enough to prove that for every $2 \leq j' \leq m-1$, if $g_{j'} =e$, there is $f' : S^1_m \rightarrow X$ such that $f'$ is a $(g_1,...,e,g_{j'-1},g_{j'+1},...,g_m)$-map and $\SFill_1 (f) \leq \SFill_1 (f')+ 2$. Note that there is some $i' \in \I$ such that $e, g_{j'-1} \in K_{\lbrace i' \rbrace}$ and therefore the needed assertion follows from (1) in Lemma \ref{relation moves lemma}.
\end{proof}

\begin{corollary}
\label{Cancellation move coro}
Let $f : S^1_m \rightarrow X$ be a $(g_1,...,g_m)$-map where  $(g_1,...,g_m) \in \left( \bigcup_{i} K_{\lbrace i \rbrace} \right)^m$ and $g_1 ... g_m =e$. If for some $1 \leq j' \leq m-2$, $g_{j'+1} = g_{j'}^{-1}$, then there is $f' : S^1_m \rightarrow X$ such that $f'$ is a $(g_1,...g_{j'-1},e,e,g_{j'+2},...,g_m)$-map and $\SFill_1 (f) \leq \SFill_1 (f')+ 2$.
\end{corollary}

\begin{proof}
This follows immediately from Lemma \ref{relation moves lemma}.
\end{proof}

\begin{lemma}
\label{identity reduction lemma}
Let $f : S^1_m \rightarrow X$ be a $(e,g_2,...,g_m)$-map where  $(e,g_2,...,g_m) \in \left( \bigcup_{i} K_{\lbrace i \rbrace} \right)^{m-1}$ and $e g_2 ... g_m =e$. Then there is a $(g_2,...,g_m)$-map $f' : S^1_{m-1} \rightarrow X$ such that
$\SFill_1 (f) \leq \SFill_1 (f')+ 2$.
\end{lemma}

\begin{proof}
Define $f'' : S^1_{m} \rightarrow X$ by
$$f'' (j) = \begin{cases}
f (1) & j = 0 \\
f(j) & 0<j
\end{cases}.$$
Note that $f '' (0) = f (1) = K_{i_1}$ and $f'' (m-1) = f (m-1) = K_{\lbrace i_{m-1} \rbrace}$, therefore $e \in f '' (0) \cap f'' (m-1)$ and thus $\lbrace f '' (0), f'' (m-1) \rbrace \in X$ and $f''$ is simplicial. Also note that
$$e \in f '' (0) \cap f (0) \cap f(1) \cap f (m-1)$$
and thus $f$ and $f''$ are contiguous. It follows from Lemma \ref{contiguous maps lemma} that $\SFill_1 (f) \leq \SFill_1 (f'')+ 2$. We finish by defining
$f' : S^1_{m-1} \rightarrow X$ by $f' (j) = f'' (j+1)$ and noting that $\SFill_1 (f'') = \SFill_1 (f')$.
\end{proof}

After this lemmas, we can prove the following Theorem that summarizes all the homotopy moves on $(g_1,...,g_m)$-maps:
\begin{theorem}[Homotopy moves Theorem]
\label{Homotopy moves Theorem}
Let $f : S^1_m \rightarrow X$ be a $(g_1,...,g_m)$-map and assume that $m >2$.
\begin{enumerate}
\item \textbf{Identity reduction move:} If for some $1 \leq j' \leq m$, $g_{j'} =e$, then there is $f' : S^1_{m-1} \rightarrow X$ such that $f'$ is a $(g_1,...g_{j'-1},g_{j'+1},...,g_m)$-map and $\SFill_1 (f) \leq \SFill_1 (f')+ 4+2(j'-1) \leq \SFill_1 (f')+ 4+2(m-1)$.
\item \textbf{Free reduction move:} If for some $1 \leq j' \leq m-2$, $g_{j'+1} = g_{j'}^{-1}$, then there is $f' : S^1_{m-2} \rightarrow X$ such that $f'$ is a $(g_1,...g_{j'-1},g_{j'+2},...,g_m)$-map and $\SFill_1 (f) \leq \SFill_1 (f')+ 6+4(j'-1) \leq \SFill_1 (f')+ 6+4(m-1)$.
\item \textbf{Relation move 1:} If for some $1 \leq j' \leq m-1$, there is $i' \in I$ such that $g_{j'}, g_{j'+1}, g_{j'} g_{j'+1} \in K_{\lbrace i' \rbrace}$ then there is $f' : S^1_{m-1} \rightarrow X$ such that $f'$ is a $(g_1,...,g_{j'-1},g_{j'} g_{j'+1},g_{j'+2},...,g_m)$-map and $\SFill_1 (f) \leq \SFill_1 (f')+ 4+2(j'-1)\leq \SFill_1 (f')+ 4+2(m-1)$.
\item \textbf{Relation move 2:} If for some $1 \leq j' \leq m-1$, there is $i' \in I$ such that $g_{j'} \in K_{\lbrace i' \rbrace}$ and $g_{j'} = g_{j'}' g_{j'}''$ with $ g_{j'}', g_{j'}'' \in K_{\lbrace i' \rbrace}$, then there is $f' : S^1_{m+1} \rightarrow X$ such that $f'$ is a $(g_1,...g_{j'-1},g_{j'}', g_{j'}'',g_{j'+1},...,g_m)$-map and $\SFill_1 (f) \leq \SFill_1 (f')+ 2$.
\end{enumerate}
\end{theorem}

\begin{proof}
\textbf{Identity reduction move:} Combine Corollary \ref{Identity ordering coro} and Lemma \ref{identity reduction lemma}.

\textbf{Free reduction move:} Combine Corollary \ref{Cancellation move coro}, Corollary \ref{Identity ordering coro} and Lemma \ref{identity reduction lemma}.

\textbf{Relation move 1:} Combine Lemma \ref{relation moves lemma}, Corollary \ref{Identity ordering coro} and Lemma \ref{identity reduction lemma}.

\textbf{Relation move 2:} Assume $g_{j'} = g_{j'}', g_{j'}''$ and $g_{j'}, g_{j'}', g_{j'}'' \in K_{\lbrace i' \rbrace}$. Define first $f'' :S^1_{m+1} \rightarrow X$ by
$$f'' (j) = \begin{cases}
f(j) & j \leq j' \\
f(j-1) & j > j'
\end{cases}.$$
It is easy to see that $f''$ is simplicial and that $\SFill_1 (f) = \SFill_1 (f'')$. Define $f' :S^1_{m+1} \rightarrow X$ by
$$f' (j) = \begin{cases}
f'' (j) & j \neq j' \\
g_1...g_{j'}' K_{\lbrace i' \rbrace} & j = j'
\end{cases}.$$
Note that
$$f' (j'+1) = f'' (j'+1) = f' (j') = g_1...g_{j'} K_{\lbrace i_{j'} \rbrace} = g_1...g_{j'}' g_{j'}'' K_{\lbrace i_{j'} \rbrace}.$$
Thus, if we show that $f'$ is simplicial, it will be a $(g_1,...g_{j'-1},g_{j'}', g_{j'}'',g_{j'+1},...,g_m)$-map. As above, we will show that $f'$ is simplicial and that $f''$ and $f'$ are contiguous and it will follow from Lemma \ref{contiguous maps lemma} that
$$\SFill_1 (f) = \SFill_1 (f'') \leq \SFill_1 (f') +2.$$
Note that $g_{j'}',g_{j'}'' \in  K_{\lbrace i' \rbrace}$ implies that
$$f' (j') = g_1...g_{j'}' K_{\lbrace i' \rbrace} = g_1...g_{j'-1} K_{\lbrace i' \rbrace},$$
and
$$f' (j') = g_1...g_{j'}' g_{j'}'' K_{\lbrace i' \rbrace} = g_1...g_{j'} K_{\lbrace i' \rbrace}.$$
Thus $f' (j') \cap f' (j'-1)$ and $f' (j') \cap f' (j'+1) = f' (j') \cap f'' (j'+1) = f' (j') \cap f'' (j')$ are non-empty and it follows that $f'$ is simplicial and that $f''$ and $f'$ are contiguous (we use the fact that $X$ is a clique complex).
\end{proof}

Observe that the homotopy moves described in this Theorem correspond to reducing the word $w = g_1 ... g_m$ to the trivial word using relations from $\bigcup_{i} R_i$ and free reductions. We will use this Theorem we are ready to prove Theorem \ref{bound on N1 thm}:
\begin{proof}[Proof of Theorem \ref{bound on N1 thm}]
By Lemma \ref{translation to (g_1,...,g_m) lemma}, it is enough to prove the Theorem for $(g_1,...,g_m)$-maps.

We note that for every simplicial map $f : S^1_m \rightarrow X$, it holds that if $m = 2$, then $\SFill_1 (f) \leq 1$. Thus, we will assume that $m >2$.

Let $f : S^1_m \rightarrow X$ be a $(g_1,...,g_m)$-map. Some of the $g_i's$ might be equal to $e$ and we start by preforming identity reduction moves on $f$ to eliminate them. Indeed, by preforming identity reduction moves (if needed), we pass to a map $f^{(0)} : S^1_{m_0} \rightarrow X$ such that $m_0 \leq m$, $f^{(0)}$ is a $(g_1^{(0)},...,g_m^{(0)})$-map, $g_1^{(0)},..., g_{m_0}^{(0)} \neq e$ and
$$\SFill_1 (f) \leq \SFill_1 (f^{(0)}) + (m-m_0)(4+2(m-1)) \leq 4m + 2m^2.$$
If $m_0 \leq 2$, then it follows that $\SFill_1 (f) \leq 2m^2 + 4m +1$ and we are done, so we will assume that $m_0 >2$.

We define the following reduction algorithm on $f^{(0)}$: Note that $w=g_1^{(0)} ... g_m^{(0)}$ is a trivial word and thus there is an algorithm for reducing it to the trivial map using the relations in $\bigcup_i R_{i}$ and free reductions. We call this the reduction algorithm of $w$ and denote by $w_k$ the word after the $k$-step in the algorithm. We note that each relation in $\bigcup_i R_{i}$ either increases the word length of $w$ by $1$ or it decreases it by $1$ and a free reduction move reduces the length of $w$ by $2$. Use the reduction algorithm of $w$ to define a reduction algorithm of $f^{(0)}$: Let $f^{(k)} : S^1_{m_k} \rightarrow X$ be an ordered $(g_1^{(k)},...,g_{m_k}^{(k)})$-map, where $w_k = g_1^{(k)} ... g_{m_k}^{(k)}$. Note that every reduction step of $w$ corresponds to a homotopy move in Theorem \ref{Homotopy moves Theorem}, thus the difference between $\SFill_1 (f^{(k)})$ and $\SFill_1 (f^{(k+1)})$ can be bounded using this Theorem.

Note that since we are using relations of length $3$, if follows that $\vert w_k \vert \leq \vert w_k \vert +1$ and thus for every $k$, $m_k \leq m + \Dehn (m)$. It follows from Theorem \ref{Homotopy moves Theorem}, that for every $k$,
$$\SFill_1 (f^{(k)}) \leq \SFill_1 (f^{(k+1)}) + 6+4(m_k-1) \leq 6 + 4 (m + \Dehn (m)).$$
Let $k_{\text{stop}}$ be the smallest number such that $\vert w_{k_{\text{stop}}} \vert \leq 2$. By the fact stated regarding the Dehn function, $k_{\text{stop}} \leq 4\Dehn (m) + m$. Also, $\SFill_1 (f^{(k_{\text{stop}})}) \leq 1$. Thus
\begin{dmath*}
\SFill_1 (f^{(0)}) \leq  k_{\text{stop}} (6 + 4 (m + \Dehn (m))) + \SFill_1 (f^{(k_{\text{stop}})}) \leq (4\Dehn (m) + m) (6 + 4 (m + \Dehn (m))) + 1 = 16 \Dehn (m)^2 + 4 m^2 + 20 m \Dehn (m) + 24 \Dehn (m) + 6 m +1.
\end{dmath*}

Recall that $\SFill_1 (f) \leq \SFill_1 (f^{(0)}) + 2m^2 + 4m$ and therefore is follows that
$$\SFill_1 (f) \leq 16 \Dehn (m)^2 + 6 m^2 + 20 m \Dehn (m) + 24 \Dehn (m) + 10 m +1.$$
\end{proof}

\begin{remark}
As the reader might have noted, we did not optimize the bounds in proofs Theorems \ref{Homotopy moves Theorem}, \ref{bound on N1 thm}, so the polynomial $p(x,y)$ in Theorem \ref{bound on N1 thm} if far from being a tight bound.
\end{remark}

Combining Corollary \ref{bound on N_0 coro}, Theorem \ref{bound on N1 thm} and Theorem \ref{filling constants criterion for coboundary exp thm2} yields the following Theorem (that generalizes Theorem \ref{N_0 + N_1 bound thm intro} that appeared in the introduction):

\begin{theorem}
\label{N_0 + N_1 bound thm}
Let $G$ be a group with subgroups $K_{\lbrace i \rbrace}, i \in \I$, where $\I$ is a finite set and $\vert I \vert \geq 2$. Denote $X = X(G, (K_{\lbrace i \rbrace})_{i \in \I})$ and let $N_0, N_1$ be the constants defined in Theorem \ref{simplicial filling constants thm}. For every $i \in I$, denote $R_{i}$ to be all the non-trivial relations in the multiplication table of $K_{\lbrace i \rbrace}$, i.e., all the relations of the form $g_1 g_2 g_3 =e$, where $g_1,g_2, g_3 \in K_{\lbrace i \rbrace} \setminus \lbrace e \rbrace$.

Assume that $K_{\lbrace i \rbrace}, i \in \I$ boundedly generate $G$ and that $G = \langle \bigcup_{i} K_{\lbrace i \rbrace} \vert \bigcup_{i} R_i \rangle$. Let $\Dehn$ denote the Dehn function of $G$ with respect to this presentation of $G$.

Then:
\begin{enumerate}
\item The constant $N_0$ is bounded by
$$N_0 ' = 1+\max_{g \in G} \min \left\lbrace l : g = g_1...g_l \text{ and } g_1,...,g_l \in \bigcup_i K_{\lbrace i \rbrace} \right\rbrace.$$
Also, if $G$ acts strongly transitively on $X$, then
$$\Exp^0_b (X) \geq \frac{1}{(n+1) N_0 '} $$
\item The constant $N_1$ is bounded by $p(2 N_0  + 1, \Dehn (2 N_0 +1))$,
where $p(x,y)$ is the polynomial
$$p(x,y) = 16 y^2 + 6 x^2 +20 xy + 24 y+ 10x +1.$$
Also, if $G$ acts strongly transitively on $X$, then
$$\Exp^1_b (X) \geq \frac{1}{{n+1 \choose 2} p(2 N_0  + 1, \Dehn (2 N_0 +1))} ,$$

\end{enumerate}
\end{theorem}


\section{New coboundary expanders}
\label{New coboundary expanders sec}

The aim of the section is to prove new examples of coboundary expanders. While our main motivation is proving Theorem \ref{new coboundary expanders thm intro} stated in the introduction. We start with the simplified case of a coset complex arising from a unipotent group over finite field. We then use this case to consider the more general case of a coset complex arising from unipotent group with polynomial entries (and those are the complexes that appear in Theorem \ref{new coboundary expanders thm intro}).

\subsection{New coboundary expanders arising from unipotent group over finite field}

Let $n \in \mathbb{N}, n \geq 2$ and $q$ a prime power. For $1 \leq i, j \leq n+1, i \neq j$ and $a \in \mathbb{F}_q$, let $e_{i,j} (a)$ be the $(n+1) \times (n+1)$ (elementary) matrix with $1$'s along the main diagonal, $a$ in the $(i,j)$ entry and $0$'s in all the other entries. The unipotent group $\Unip_{n+1} (\mathbb{F}_q)$ is the group generated by
$$\lbrace e_{i,j} (a) : a \in \mathbb{F}_q, 1 \leq i < j \leq n+1 \rbrace,$$
and it is easy to verify that $\Unip_{n+1} (\mathbb{F}_q)$ is in fact the group of upper triangular $(n+1) \times (n+1)$ matrices with entries in $\mathbb{F}_q$.

For $0 \leq i \leq n-1$, define a subgroup $K_{\lbrace i \rbrace} <\Unip_{n+1} (\mathbb{F}_q)$ as
$$K_{\lbrace i \rbrace} = \langle e_{j, j+1} (a) : j \in \lbrace 1,...,n \rbrace \setminus \lbrace i+1 \rbrace, a \in \mathbb{F}_q \rangle.$$
The aim of this section is to show that if $n \geq 4$ or $q$ is odd, then for $X = X(\Unip_{n+1} (\mathbb{F}_q), (K_{\lbrace i \rbrace})_{i =0,...,n-1})$,  $\Exp_b^0 (X), \Exp_b^1 (X)$ are bounded from below independently of $q$.

\begin{lemma}
\label{warm-up - bound on N_0 - lemma}
The subgroup $K_{\lbrace 0 \rbrace},...,K_{\lbrace n-1 \rbrace}$ boundedly generate $\Unip_{n+1} (\mathbb{F}_q)$ and in fact every $g \in \Unip_{n+1} (\mathbb{F}_q)$ can be written as a product of at most $6$ elements in $K_{\lbrace 0 \rbrace} \cup K_{\lbrace n-1 \rbrace}$.
\end{lemma}

\begin{proof}
Let $g \in \Unip_{n+1} (\mathbb{F}_q)$. By Gauss elimination, there are $g_1 \in K_{\lbrace 0 \rbrace}, g_2 \in K_{\lbrace n-1 \rbrace}$ and $a \in \mathbb{F}_q$ such that $g_2^{-1} g^{-1}_1 g = e_{1,n+1} (a)$.  Note that
$$e_{1, n+1} (a) = e_{1, n} (-a)  e_{n, n+1} (-1) e_{1, n} (a)  e_{n, n+1} (1),$$
and thus it is a product of $4$ elements in $K_{\lbrace 0 \rbrace} \cup K_{\lbrace n-1 \rbrace}$.
\end{proof}

\begin{corollary}
\label{warm-up - bound on N_0 coro}
For every prime power $q$, $ n\geq 2$, let $G = \Unip_{n+1} (\mathbb{F}_q), K_{\lbrace 0 \rbrace},...,K_{\lbrace n-1 \rbrace}$ as above and $X = X(G, (K_{\lbrace i \rbrace})_{i \in \lbrace 0,...,n-1 \rbrace})$. The filling constant $N_0$ of $X$ satisfies $N_0 \leq 7$ and in particular is bounded independently of $q$.
\end{corollary}

\begin{proof}
Combine Lemma \ref{warm-up - bound on N_0 - lemma} and Corollary \ref{bound on N_0 coro}.
\end{proof}

After bounding $N_0$, we want to apply Theorem \ref{N_0 + N_1 bound thm} in order to bound $N_1$. Thus we need to show that the group $G= \Unip_{n+1} (\mathbb{F}_q)$ can be presented as
$$G = \langle \bigcup_{i=0}^{n-1} K_{\lbrace i \rbrace} \vert \bigcup_{i=0}^{n-1} R_{i} \rangle,$$
where $R_i$ are all the relations of $K_{\lbrace i \rbrace}$ and that the Dehn function for this presentation can be bounded independently of $q$.

We start by introducing a known set of relations for $G$. We recall that for a group $G$ and $g, h \in G$, the commutator $[g, h]$ is defined by $[g, h] = g^{-1} h^{-1} g h$.
Fix $n \geq 2$ and for $1 \leq i < j \leq n+1$, $a \in \mathbb{F}_q$, denote by $e_{i,j} (a)$ the elementary matrix defined above. The Steinberg relations of $G$ are the following (we leave it to the reader to verify that these relations holds):
\begin{enumerate}[label=(St {{\arabic*}})]
\item For every $1 \leq i < j \leq n+1$, and every $a_1, a_2 \in \mathbb{F}_q$,
$$e_{i,j} (a_1) e_{i,j} (a_2) = e_{i,j} (a_1 + a_2).$$
\item For every $1 \leq i_1 < j_1 \leq n+1, 1 \leq i_2 < j_2 \leq n+1$, and every $a_1, a_2 \in \mathbb{F}_q$,
$$[e_{i_1,j_1} (a_1),e_{i_2,j_2} (a_2)] =
\begin{cases}
1 & i_1 \neq j_2, i_2 \neq j_1 \\
e_{i_1, j_2} (a_1 a_2) & j_1 = i_2 \\
e_{i_2, j_1} (- a_1 a_2) & j_2 = i_1
\end{cases}.$$
\end{enumerate}

\begin{lemma}
\label{warm-up - Steinberg lemma}
Let $G =  \Unip_{n+1} (\mathbb{F}_q)$, $S_{el}$ be the set of all elementary matrices  and $R_{St}$ be the Steinberg relations. Then $\langle S_{el}  \vert R_{Steinberg} \rangle = G$ and the Dehn function for this presentation is bounded independently of $q$.
\end{lemma}

\begin{proof}
We observe that every $g \in G$ can be written as a product of the form
$$g =  e_{1,2} (a_{1,2}) e_{1,3} (a_{1,3}) ... e_{n, n+1} (a_{n,n+1}),$$
where $a_{i,j} \in \mathbb{F}_q$. Moreover, $g = e$ if and only if $a_{i,j} = 0$ for every $1 \leq i < j \leq n+1$. Thus, it is enough to prove that any word $w$ of length $m$ can be brought to this form using the Steinberg relations and that the number of relations applied is independent of $q$.

Let $w$ be a word of length $m$. We bring $w$ to the desired form using the Steinberg relations above to preform bubble sort. First, we can always write $w$ as $w = w_1 e_{1,2} (a_{1,2}^{(1)}) w_2 e_{1,2} (a_{1,2}^{(2)}) ... w_{k} e_{1,2} (a_{1,2}^{(k)}) w_{k+1}$ where $w_i$ do not contain any $e_{1,2}$'s element. We use the Steinberg relations (St 2) to permute the $e_{1,2}$ and relation (St 1) to merge them and bring $w$ to the form $e_{1,2} (a_{1,2}^{(1)} + ...+a_{1,2}^{(k)}) w'$ where $e_{1,2} (a_{1,2}) = e_{1,2} (a_{1,2}^{(1)} + ...+a_{1,2}^{(k)})$ and $w' = w_1' ... w_k ' w_{k+1}$. Note that the number of relations that are applied is $\leq 2m$ (at most $m$ ``bubble'' permutations and at most $m$ merges). Also note that while $w'$ may be longer than $w$ (since applying relation (St 2) can add a letter to the word) its length is at most $2m$ and it does not contain any elements of the form $e_{1,2}$. Thus, ordering $e_{1,2}$ ``cost'' us at most $2m$ relations and made the prefix at most $2$ times longer. Repeat the same algorithm on $w'$ to order $e_{1,3}$: since $w'$ is at length $\leq 2m$, this will ``cost'' at most $4m$ applications of the relations and make the word into $w'' = e_{1,2} (a_{1,2}) e_{1,3} (a_{1,3}) w''$, where the length of $w''$ is at most $4m$ and is contains no $e_{1,2}, e_{1,3}$'s. Repeating this ordering procedure for all the $e_{i,j}$'s we are able to bring $w$ to the desired form applying at most $2m + 4m + ... + 2^{{n+1} \choose 2}$ relations and this number does not depend on $m$. Thus if we denote $\Dehn_{St}$ to be the is the Dehn function with respect to the Steinberg relations (with the generating set of elementary matrices), we proved that $\Dehn_{St} (m) \leq 2m + 4m + ... + 2^{{n+1} \choose 2} m$ and this number does not depend on $q$. We note that the bound we gave on $\Dehn_{St}$ is far from tight, since we made no to effort to optimize this bound, but only show it is independent of $q$.
\end{proof}

Note that every $g \in \bigcup_i K_{\lbrace i \rbrace}$ can be written as a product of at most ${n \choose 2}$ and thus from the above Lemma we can conclude that in order to show that
$$G = \langle \bigcup_{i=0}^{n-1} K_{\lbrace i \rbrace} \vert \bigcup_{i=0}^{n-1} R_{i} \rangle$$
and that the Dehn function for this presentation can be bounded independently of $q$, it is sufficient to show that every Steinberg relation can be reduced to the trivial word using $\bigcup_i R_i$ and that the number of relations in $\bigcup_i R_i$ that are applied in such a reduction is bounded independently of $q$. Formally, when proving that
$$G = \langle \bigcup_{i=0}^{n-1} K_{\lbrace i \rbrace} \vert \bigcup_{i=0}^{n-1} R_{i} \rangle$$
we can not refer to elements of the form $e_{1,n+1} (a)$ since they do not appear in our set of generators. The way to overcome this problem is to formally define $e_{1,n+1} (a)$ as $e_{1,n+1} (a) = [e_{1,n} (1), e_{n,n+1} (a)]$. The following observation states that we do not have to consider all the Steinberg relations, since some already appear in $\bigcup_i R_i$:

\begin{observation}
\label{rel obsrv}
We note that a lot of Steinberg relations already appear in $\bigcup_i R_i$. Namely, if we denote $e_{1,n+1} (a) = [e_{1,n} (1), e_{n,n+1} (a)]$ the Steinberg relations that do \textbf{not} appear in $R_0 \cup R_{n-1}$ are:
\begin{enumerate}
\item For every $1 < j \leq n, 2 \leq i < n+1$, $i \neq j$ and every $a, b \in \mathbb{F}_q$
$$[e_{1,j} (a),e_{i,n+1} (b)] = 1.$$
\item For every $1 < j < n+1$ and every $a, b \in \mathbb{F}_q$,
$$[e_{1,j} (a),e_{j,n+1} (b)] = e_{1,n+1} (a b).$$
\item For every $1 \leq i < j  \leq n+1$ and every $a, b \in \mathbb{F}_q$,
$$[e_{i,j} (a),e_{1,n+1} (b)] = 1.$$
\item For every $a, b \in \mathbb{F}_q$,
$$e_{1,n+1} (a) e_{1,n+1} (b) = e_{1,n+1} (a + b).$$
\end{enumerate}
\end{observation}

In \cite{BissD}, Biss and Dasgupta proved the following:
\begin{theorem}\cite[Theorem 1]{BissD}
\label{BissD thm 1}
Let $n \geq 3$ and $G = \Unip_{n+1} (\mathbb{F}_q)$. For $1 \leq i \leq n$, let $s_1,...,s_n$ be the generators of the abstract group $G_0$ and define the following relations:
\begin{enumerate}[label=(B-D {{\arabic*}})]
\item For every $1 \leq i \leq n$, $s_i^q = 1$.
\item For every $1 \leq i \leq n-2$ and every $j > i-1$, $[s_i,s_j] =1$.
\item For every $1 \leq i \leq n-1$,
$$[s_i, [s_i,s_{i+1}]]= [s_{i+1}, [s_i,s_{i+1}]] =1.$$
\item For every $1 \leq i \leq n-2$,
$$[[s_{i}, s_{i+1}], [s_{i+1}, s_{i+2}]] =1.$$
\end{enumerate}
Then $G_0$ which is $\lbrace s_i^{\pm} : 1 \leq i \leq n \rbrace$ with the relations $(1)-(4)$ is isomorphic to $G$. Moreover, if $q$ is odd, the relations of the form $(4)$ are not needed.

Explicitly, the proof of \cite[Theorem 1]{BissD} shows the following: Define in $G_0$ (by abuse of notation) $e_{i,j} (1)$ inductively as $e_{i,j} (1) = s_i$ if $j =i+1$ and $e_{i,j} (1) = [e_{i,j-1} (1), s_{j-1}]$. Then every relation of the form
\begin{enumerate}
\item For every $1 \leq i_1 < j_1 \leq n+1, 1 \leq i_2 < j_2 \leq n+1$,
$$[e_{i_1,j_1} (1),e_{i_2,j_2} (1)] =
\begin{cases}
1 & i_1 \neq j_2, i_2 \neq j_1 \\
e_{i_1, j_2} (1) & j_1 = i_2 \\
e_{i_2, j_1} (- 1) & j_2 = i_1
\end{cases}.$$
\item For every $1 \leq i_1 < i_2 < i_3 < i_4$,
$$[e_{i_1,i_2} (1),[e_{i_2,i_3} (1), e_{i_3,i_4} (1)]] = [[e_{i_1,i_2} (1),e_{i_2,i_3} (1)], e_{i_3,i_4} (1)].$$
\item For every $1 \leq l<i <j \leq n+1$,
$$[e_{l,i} (1), e_{i,j} (1)] = e_{l,j} (1).$$
\end{enumerate}
can be deduced from $(1)-(4)$ (or $(1)-(3)$ if $q$ is odd) in a finite number of steps.
\end{theorem}

\begin{corollary}
\label{BissD coro 1}
Let $n \geq 3$ and $G = \Unip_{n+1} (\mathbb{F}_q)$. Fix $a_1,...,a_{n} \in \mathbb{F}_q \setminus \lbrace 0 \rbrace$. For $1 \leq i \leq n$, denote $s_i (a_i) = e_{i,i+1} (a_i)$ and consider the following relations in $G$:
\begin{enumerate}
\item For every $1 \leq i \leq n$, $s_i (a_i)^q = 1$.
\item For every $1 \leq i \leq n-2$ and every $j > i-1$, $[s_i (a_i),s_j (a_j) ] =1$.
\item For every $1 \leq i \leq n-1$,
$$[s_i (a_i), [s_i (a_i),s_{i+1} (a_{i+1})]]= [s_{i+1} (a_{i+1}), [s_i (a_{i}),s_{i+1} (a_{i+1})]] =1.$$
\item For every $1 \leq i \leq n-2$,
$$[[s_{i} (a_{i}), s_{i+1} (a_{i+1})], [s_{i+1} (a_{i+1}), s_{i+2} (a_{i+2})]] =1.$$
\end{enumerate}
Denote $a_i^j = \prod_{l=i}^{j -1} a_l$ and define $e_{i,j} (a_i^j)$ inductively as $e_{i,i+1} (a_i^{i+1}) = s_i (a_i)$ if and $e_{i,j} (a_i^j) = [e_{i,j-1} (a_i^{j-1}), s_{j-1} (a_{j-1})]$. Then every relation of the form
\begin{enumerate}
\item For every $1 \leq i_1 < j_1 \leq n+1, 1 \leq i_2 < j_2 \leq n+1$,
$$[e_{i_1,j_1} (a_{i_1}^{j_1}),e_{i_2,j_2} (a_{i_2}^{j_2})] =
\begin{cases}
1 & i_1 \neq j_2, i_2 \neq j_1 \\
e_{i_1, j_2} (a_{i_1}^{j_1} a_{i_2}^{j_2}) & j_1 = i_2 \\
e_{i_2, j_1} (- a_{i_1}^{j_1} a_{i_2}^{j_2}) & j_2 = i_1
\end{cases}.$$
\item For every $1 \leq i_1 < i_2 < i_3 < i_4$,
$$[e_{i_1,i_2} (a_{i_1}^{i_2}),[e_{i_2,i_3} (a_{i_2}^{i_3}), e_{i_3,i_4} (a_{i_3}^{i_4})]] = [[e_{i_1,i_2} (a_{i_1}^{i_2}),e_{i_2,i_3} (a_{i_2}^{i_3})], e_{i_3,i_4} (a_{i_3}^{i_4})].$$
\item For every $1 \leq l<i <j \leq n+1$,
$$[e_{l,i} (a_l^{i}), e_{i,j} (a_i^{j})] = e_{l,j} (a_l^{j}).$$
\end{enumerate}
can be deduced from $(1)-(4)$ (or $(1)-(3)$ if $q$ is odd) in a finite number of steps (that is independent of $q$).
\end{corollary}

\begin{proof}
Consider the isomorphism $\Psi : G \rightarrow G$ defined by $\Psi (s_{i} (a_i)) = s_{i} (1)$ and apply Theorem \ref{BissD thm 1} on $\Psi (G)$.
\end{proof}

After this Corollary, we are ready to prove the following Theorem:
\begin{theorem}
\label{warm-up bound on N_1 thm}
Let $G =  \Unip_{n+1} (\mathbb{F}_q)$, with $K_{\lbrace 0 \rbrace},..., K_{\lbrace n-1 \rbrace}, R_0,...,R_{n-1}$ as above. If $n \geq 4$ or $q$ is odd, then
$$G = \langle \bigcup_i K_{\lbrace i \rbrace} \vert \bigcup_i R_i \rangle$$
and the Dehn function of this presentation is bounded independently of $q$.
\end{theorem}

\begin{proof}
First, we note that all the relations $(B-D 1)-(B-D 4)$ of Theorem \ref{BissD thm 1} appear in $\bigcup_i R_i$ and therefore $G = \langle \bigcup_i K_{\lbrace i \rbrace} \vert \bigcup_i R_i \rangle$. Thus we already know all the Steinberg relations holds in $\langle \bigcup_i K_{\lbrace i \rbrace} \vert \bigcup_i R_i \rangle$ and we just need to verify that every Steinberg relation can be deduced from a finite number of relations in $\bigcup_i R_i$ that is independent of $q$. As stated in Observation \ref{rel obsrv}, we do not have to check all the Steinberg relations (since some appear in  $\bigcup_i R_i$) and it is enough to check relations $(1)-(4)$ that were stated in Observation \ref{rel obsrv}.

\textbf{Observation \ref{rel obsrv}, relations of type (1):} We need to show that $\bigcup_i R_i$ imply that for every $1 < j \leq n, 2 \leq i < n+1$, $i \neq j$ and every $a, b \in \mathbb{F}_q$
$$[e_{1,j} (a),e_{i,n+1} (b)] = 1.$$
Apply Corollary \ref{BissD coro 1} with $a_1 = a$, $a_2 = ... = a_{n-1} =1$, $a_n =b$. Then for every $j \leq n$, $a_1^j =a$ and every $2 \leq i$, $a_i^{n+1} = b$ and from Corollary \ref{BissD coro 1}, $\bigcup_i R_i$ imply that
$$[e_{1,j} (a),e_{i,n+1} (b)] = 1,$$
and the number of relations for deducing this relation is independent of $q$.

\textbf{Observation \ref{rel obsrv}, relations of type (2):} We need to show that $\bigcup_i R_i$ imply that for every $1 < j < n+1$ and every $a, b \in \mathbb{F}_q$,
$$[e_{1,j} (a),e_{j,n+1} (b)] = e_{1,n+1} (a b),$$
where $e_{1,n+1} (a b)$ is formally defined as
$$e_{1,n+1} (a b) = [e_{1,n} (1), e_{n,n+1} (ab)].$$
We will first show this with $b=1$. Let $a_1 =...=a_{j-2} =1, a_{j-1} =a$ and $a_j=...=a_n =1$. Then $a_1^j = a, a_j^{n+1} =1, a_1^{n+1} =a$ and by Corollary \ref{BissD coro 1}, $\bigcup_i R_i$ imply that
$$[e_{1,j} (a), e_{j,n+1} (1)] = [e_{1,j} (a_1^{j}), e_{j,n+1} (a_j^{n+1})] = e_{l,n+1} (a_1^{n+1})= e_{1,n+1} (a)$$
and the number of relations needed for deducing this relation is independent of $q$.

Next, we will treat the case where $a,b$ are general and $j <n$. We first note that
$$[e_{j,j+1} (b), e_{j+1,n+1} (1)] = e_{j,n+1} (b),$$
is in $R_0$ and thus (after applying one relation)
$$[e_{1,j} (a),e_{j,n+1} (b)] = [e_{1,j} (a),[e_{j,j+1} (b), e_{j+1,n+1} (1)]].$$

Last, in the case where $j =n$, we use the relation $e_{1,n} (a) = [e_{1,2} (a), e_{2,n} (1)]$ and using similar arguments as above, we show that
$$[e_{1,n} (a),e_{n,n+1} (b)] = [e_{1,2} (a), e_{2,n+1} (b)],$$
and thus we can use the previous case.

\textbf{Observation \ref{rel obsrv}, relations of type (3):} We need to show that $\bigcup_i R_i$ imply that for every $1 \leq i < j  \leq n+1$ and every $a, b \in \mathbb{F}_q$,
$$[e_{i,j} (a),e_{1,n+1} (b)] = 1,$$
and that the number of relations used is independent of $q$.
We will first prove the case where $\lbrace i,j \rbrace \neq \lbrace 1,n+1 \rbrace$. Let $a, b$ as above. If $a=0$, then $e_{i,j} (a)$ is the identity and there is nothing to prove. Assume that $a \neq 0$ and  define $a_1 =a, a_2 =...=a_{n-1} =1, a_n = \frac{b}{a}$. Then $a_{i,j} =a$, $a_{1,n+1} =b$ and the relation follows from Corollary \ref{BissD coro 1}. We are left to prove that
$$[e_{1,n+1} (a),e_{1,n+1} (b)] = 1,$$
but this follows from the previous case combined with the fact that by definition  $e_{1,n+1} (a) = [e_{1,n} (1), e_{n,n+1} (a)]$.

\textbf{Observation \ref{rel obsrv}, relations of type (4):} We need to show that $\bigcup_i R_i$ imply that for every $1 \leq i < j  \leq n+1$ and every $a, b \in \mathbb{F}_q$,
$$e_{1,n+1} (a) e_{1,n+1} (b) = e_{1,n+1} (a+b),$$
and that the number of relations used is independent of $q$. By definition, this is equivalent to showing that
$$[e_{1,n} (1),e_{n,n+1} (a)][e_{1,n} (1),e_{n,n+1} (b)] = [e_{1,n} (1),e_{n,n+1} (a+b)].$$
Indeed,
\begin{dmath*}
[e_{1,n} (1),e_{n,n+1} (a+b)] =
e_{1,n} (-1) e_{n,n+1} (- a-b) e_{1,n} (1) e_{n,n+1} (a+b) =^{\text{Using the relations in } R_{0}}
e_{1,n} (-1) e_{n,n+1} (- a) e_{n,n+1} (-b) e_{1,n} (1) e_{n,n+1} (b) e_{n,n+1} (a) =
e_{1,n} (-1) e_{n,n+1} (- a) e_{1,n} (1) \left( e_{1,n} (-1) e_{n,n+1} (-b) e_{1,n} (1) e_{n,n+1} (b) \right) e_{n,n+1} (a) =
e_{1,n} (-1) e_{n,n+1} (- a) e_{1,n} (1) e_{1,n+1} (b) e_{n,n+1} (a) =^{\text{Using relations of type (3) proven above}}
e_{1,n} (-1) e_{n,n+1} (- a) e_{1,n} (1)  e_{n,n+1} (a) e_{1,n+1} (b) = e_{1,n+1} (a) e_{1,n+1} (b),
\end{dmath*}
as needed.
\end{proof}

\begin{corollary}
Let $G = \Unip_{n+1} (\mathbb{F}_q)$ with subgroups $K_{\lbrace 0 \rbrace},...,K_{\lbrace n-1 \rbrace}$ as above and let $X = X(G, (K_{\lbrace i \rbrace})_{i \in \lbrace 0,...,n-1 \rbrace})$ be the coset complex. If $n \geq 4$ or $q$ is odd then the constants $N_0$ and $N_1$ of $X$ are bounded independently of $q$ and thus $\Exp_b^0 (X), \Exp_b^1 (X)$ are bounded from below independently of $q$.
\end{corollary}

\begin{proof}
By \cite[Theorem 3.5]{KOCosetGeom} and Theorem \ref{transitive action thm} stated above, $X = X(G, (K_{\lbrace i \rbrace})_{i \in \lbrace 0,...,n-1 \rbrace})$ is strongly symmetric. Thus, combining Corollary \ref{warm-up - bound on N_0 coro}, Theorem \ref{warm-up bound on N_1 thm} and Theorem \ref{N_0 + N_1 bound thm} yields the desired result.
\end{proof}


\subsection{New coboundary expanders arising from unipotent group with polynomial entries}

Below, we will show that we can generalize the example given above and get new examples of coboundary expander from unipotent groups with polynomial entries.

Define the group $G$ to be a subgroup of $(n+1) \times (n+1)$ invertible matrices with entries in $\mathbb{F}_q [t]$ in generated by the set
$\lbrace e_{i,i+1} (a +bt) : a,b \in \mathbb{F}_q, 1 \leq i \leq n \rbrace$. More explicitly, an $(n+1) \times (n+1)$ matrix $A$ is in $G$ if and only if
$$A (i,j) = \begin{cases}
1 & i=j \\
0 & i>j \\
a_0 + a_1 t + ... + a_{j-i} t^{j-i} & i<j, a_0,...,a_{j-i} \in \mathbb{F}_q
\end{cases},$$
(observe that all the matrices in $G$ are upper triangular).

For $0 \leq i \leq n-1$, define a subgroup $K_{\lbrace i \rbrace} <G$ as
$$K_{\lbrace i \rbrace} = \langle e_{j, j+1} (a+bt) : j \in \lbrace 1,...,n \rbrace \setminus \lbrace i+1 \rbrace, a,b \in \mathbb{F}_q \rangle.$$

Define $X$ to be the coset complex $X=X(G, (K_{\lbrace i \rbrace})_{i \in \lbrace 1,...,n-1 \rbrace})$. By applying Theorem \ref{N_0 + N_1 bound thm}, we will prove that for any odd $q$, the $0$-dimensional and the $1$-dimensional Cheeger constants of $X$ can be bounded away from $0$ and this bound is independent of $q$.

The bound on $N_0$ of $X$ follows from the following Lemma:

\begin{lemma}
\label{bound on N_0 - lemma}
Let $q$ be a prime power, $n \geq 2$ and $G, K_{\lbrace 0 \rbrace},...,K_{\lbrace n-1 \rbrace}$ be the groups defined above. Then the subgroups $K_{\lbrace 0 \rbrace},...,K_{\lbrace n-1 \rbrace}$ boundedly generate $G$ and
$$\max_{g \in G} \min \left\lbrace l : g = g_1...g_l \text{ and } g_1,...,g_l \in \bigcup_i K_{\lbrace i \rbrace} \right\rbrace \leq 2 + 4 (n+1).$$
\end{lemma}

The proof of this Lemma is very similar to the proof of Lemma \ref{warm-up - bound on N_0 - lemma}:
\begin{proof}
Let $g \in G$ be some group element. By Gauss elimination, there are $g_1 \in K_{\lbrace 0 \rbrace}, g_2 \in K_{\lbrace n-1 \rbrace}$ such that
$$g_2^{-1} g_1^{-1} g = e_{1, n+1} (a_0 + a_1 + ... + a_n t^n) = e_{1, n+1} (a_0) e_{1, n+1} (a_1 t)... e_{1, n+1} (a_n t^n).$$
Thus, it is enough to show that every element of the form $e_{1, n+1} (a_j t^j)$ can be written as a product of $4$ elements in $K_{\lbrace 0 \rbrace} \cup K_{\lbrace n \rbrace}$.

For every $0 \leq j < n$ and every $a_j \in \mathbb{F}_q$, we have that $e_{1, n} (\pm a_j t^{j}) \in K_{\lbrace 0 \rbrace},  e_{n, n+1} ( \pm 1) \in K_{\lbrace n \rbrace}$ and that
$$e_{1, n+1} (a_j t^j) = e_{1, n} (-a_j t^{j})  e_{n, n+1} (-1) e_{1, n} (a_j t^{j})  e_{n, n+1} (1),$$
as needed. Also, for every $a_n \in \mathbb{F}_q$, we have that $e_{1, n+1} (\pm a_n t^{n-1}) \in K_{\lbrace 0 \rbrace}, e_{n, n+1} ( \pm t) \in K_{\lbrace n \rbrace}$ and that
$$e_{1, n+1} (a_n t^n) = e_{1, n} (-a_n t^{n-1})  e_{n, n+1} (-t) e_{1, n} (a_n t^{n-1})  e_{n, n+1} (t).$$
\end{proof}

\begin{corollary}
\label{bound on N_0 coro}
For every prime power $q$, $ n\geq 2$, let $G, K_{\lbrace 0 \rbrace},...,K_{\lbrace n \rbrace}$ as above and $X = X(G, (K_{\lbrace i \rbrace})_{i \in \lbrace 0,...,n-1 \rbrace})$. The filling constant $N_0$ of $X$ satisfies $N_0 \leq 3 + 4 (n+1)$ and in particular is bounded independently of $q$.
\end{corollary}

\begin{proof}
Combine Lemma \ref{bound on N_0 - lemma} and Corollary \ref{bound on N_0 coro}.
\end{proof}

After bounding $N_0$, we want to apply Theorem \ref{N_0 + N_1 bound thm} in order to bound $N_1$. Thus we need to show that the group $G$ can be presented as
$$G = \langle \bigcup_{i=0}^{n-1} K_{\lbrace i \rbrace} \vert \bigcup_{i=0}^{n-1} R_{i} \rangle,$$
where $R_i$ are all the relations of $K_{\lbrace i \rbrace}$ and that the Dehn function for this presentation can be bounded independently of $q$.

As in the case where $\Unip (\mathbb{F}_q)$ discussed above, the group $G$ can be presented using Steinberg relations. Fix $n \geq 2$ and for $1 \leq i < j \leq n+1$, $r \in \mathbb{F}_q [t]$ of degree $\leq j-i$, denote by $e_{i,j} (r)$ the elementary matrix defined above. The Steinberg relations of $G$ are the following (we leave it to the reader to verify that these relations holds):
\begin{enumerate}[label=(S{{\arabic*}})]
\item For every $1 \leq i < j \leq n+1$, and every $r_1, r_2 \in \mathbb{F}_q [t]$ of degree $\leq j-i$,
$$e_{i,j} (r_1) e_{i,j} (r_2) = e_{i,j} (r_1 + r_2).$$
\item For every $1 \leq i_1 < j_1 \leq n+1, 1 \leq i_2 < j_2 \leq n+1$, and every $r_1, r_2 \in \mathbb{F}_q [t]$ of degree $\leq j_1-i_1, j_2-i_2$,
$$[e_{i_1,j_1} (r_1),e_{i_2,j_2} (r_2)] =
\begin{cases}
1 & i_1 \neq j_2, i_2 \neq j_1 \\
e_{i_1, j_2} (r_1 r_2) & j_1 = i_2 \\
e_{i_2, j_1} (- r_1 r_2) & j_2 = i_1
\end{cases}.$$
\end{enumerate}

\begin{lemma}
Let $G$ be as above, $S < G$ be the set $S = \lbrace e_{i,j} (r) : 1 \leq i < j \leq n+1, r \in \mathbb{F}_q [t], \deg (r) \leq j-i \rbrace$ and $R_{Steinberg}$ be the Steinberg relations as above. Then $\langle S \vert R_{Steinberg} \rangle = G$ and the Dehn function for this presentation is bounded independently of $q$.
\end{lemma}

The proof is repeating the proof of Lemma \ref{warm-up - Steinberg lemma} almost verbatim and we give it below for completeness.
\begin{proof}
We observe that every $g \in G$ can be written as a product of the form
$$g =  e_{1,2} (r_{1,2}) e_{1,3} (r_{1,3}) ... e_{n, n+1} (r_{n,n+1}),$$
where $r_{i,j} \in \mathbb{F}_q [t]$ of degree $\leq j-i$. Moreover, $g = e$ if and only if $r_{i,j} = 0$ for every $1 \leq i < j \leq n+1$. Thus, it is enough to prove that any word $w$ of length $m$ can be brought to this form using the Steinberg relations and that the number of relations applied is independent of $q$.

Let $w$ be a word of length $m$. We bring $w$ to the desired form using the Steinberg relations above to preform bubble sort. First, we can always write $w$ as $w = w_1 e_{1,2} (r_{1,2}^{(1)}) w_2 e_{1,2} (r_{1,2}^{(2)}) ... w_{k} e_{1,2} (r_{1,2}^{(k)}) w_{k+1}$ where $w_i$ do not contain any $e_{1,2}$'s element. We use the Steinberg relations (S2) to permute the $e_{1,2}$ and relation (S1) to merge them and bring $w$ to the form $e_{1,2} (r_{1,2}^{(1)} + ...+r_{1,2}^{(k)}) w'$ where $e_{1,2} (r_{1,2}) = e_{1,2} (r_{1,2}^{(1)} + ...+r_{1,2}^{(k)})$ and $w' = w_1' ... w_k ' w_{k+1}$. Note that the number of relations that are applied is $\leq 2m$ (at most $m$ ``bubble'' permutations and at most $m$ merges). Also note that while $w'$ may be longer than $w$ (since applying relation (3) adds a letter to the word) its length is at most $2m$ and it does not contain any elements of the form $e_{1,2}$. Thus, ordering $e_{1,2}$ ``cost'' us at most $2m$ relations and made the prefix at most $2$ times longer. Repeat the same algorithm on $w'$ to order $e_{1,3}$: since $w'$ is at length $\leq 2m$, this will ``cost'' at most $4m$ applications of the relations and make the word into $w'' = e_{1,2} (r_{1,2}) e_{1,3} (r_{1,3}) w''$, where the length of $w''$ is at most $4m$ and is contains no $e_{1,2}, e_{1,3}$'s. Repeating this ordering procedure for all the $e_{i,j}$'s we are able to bring $w$ to the desired form applying at most $2m + 4m + ... + 2^{{n+1} \choose 2}$ relations and this number does not depend on $m$. Thus is $\Dehn_{Steinberg}$ is the Dehn function with respect to the Steinberg relations, we proved that $\Dehn_{Steinberg} (m) \leq 2m + 4m + ... + 2^{{n+1} \choose 2} m$ and this number does not depend on $q$. We note that the bound we gave on $\Dehn_{Steinberg}$ is far from tight, since we made no to effort to optimize this bound, but only show it is independent of $q$.
\end{proof}

Next, we define what we call \textit{pure degree} Steinberg relations that are the following subset of the Steinberg relations mentioned above:
\begin{enumerate}[label=(pdS{{\arabic*}})]
\item For every $1 \leq i < j \leq n+1$, every $0 \leq k \leq j-i$ and every $a, b\in \mathbb{F}_q$
$$e_{i,j} (a t^k) e_{i,j} (b t^k) = e_{i,j} ((a+b)t^k),$$
\item For every $1 \leq i_1 < j_1 \leq n+1, 1 \leq i_2 < j_2 \leq n+1$, every $0 \leq k_1 \leq j_1 - i_1, 0 \leq k_2 \leq j_2 - i_2$ and every $a, b\in \mathbb{F}_q$,
$$[e_{i_1,j_1} (a t^{k_1}),e_{i_2,j_2} (b t^{k_2})] =
\begin{cases}
1 & i_1 \neq j_2, i_2 \neq j_1 \\
e_{i_1, j_2} (ab t^{k_1 + k_2}) & j_1 = i_2 \\
e_{i_2, j_1} (- ab t^{k_1 + k_2}) & j_2 = i_1
\end{cases}.$$
\end{enumerate}

We observe that every Steinberg relation can be written as a finite number of conjugates of pure degree relations and thus if we consider the presentation of $G$ with respect to pure dimension Steinberg relations and the generating set $\lbrace e_{i,j} (a t^k) : a \in \mathbb{F}_q, 1 \leq i < j \leq n+1, 0 \leq k \leq j-i \rbrace$, we will get that up to a multiplicative constant that is independent of $q$ the Dehn function with respect to this presentation bounds the function $\Dehn_{Steinberg}$ discussed above. The upshot of this discussion is that in order to show that the group $G$ above can be presented as
$$G = \langle \bigcup_{i=0}^{n-1} K_{\lbrace i \rbrace} \vert \bigcup_{i=0}^{n-1} R_{i} \rangle$$
and that the Dehn function for this presentation can be bounded independently of $q$, it is sufficient to show that every pure degree Steinberg relation can be reduced to the trivial word using $\bigcup_i R_i$ and that the number of relations in $\bigcup_i R_i$ that are applied in such a reduction is bounded independently of $q$. Similar to Observation \ref{rel obsrv}, we observe that we do not have to prove all the pure degree Steinberg relations, since some already appear in $\bigcup_i R_i$:

\begin{observation}
\label{rel obsrv 2}
We note that a lot of Steinberg relations already appear in $\bigcup_i R_i$. Namely, if we denote
$$e_{1,n+1} (at^k) = \begin{cases}
[e_{1,n} (t^k), e_{n,n+1} (a)] & k < n \\
[e_{1,n} (t^{k-1}), e_{n,n+1} (at)] & k = n \\
\end{cases}$$
the pure degree Steinberg relations that do \textbf{not} appear in $R_0 \cup R_{n-1}$ are:
\begin{enumerate}
\item For every $1 < j \leq n, 2 \leq i < n+1$, $i \neq j$, every $0 \leq k_1 \leq j-1, 0 \leq k_2 \leq n+1-i$ and every $a, b \in \mathbb{F}_q$
$$[e_{1,j} (a t^{k_1}),e_{i,n+1} (b t^{k_2})] = 1.$$
\item For every $1 < j < n+1$, every $0 \leq k_1 \leq j-1, 0 \leq k_2 \leq n+1-j$ and every $a, b \in \mathbb{F}_q$,
$$[e_{1,j} (at^{k_1}),e_{j,n+1} (b t^{k_2})] = e_{1,n+1} (a b t^{k_1 + k_2}).$$
\item For every $1 \leq i < j  \leq n+1$, every $0 \leq k_1 \leq j-i, 0 \leq k_2 \leq n$ and every $a, b \in \mathbb{F}_q$,
$$[e_{i,j} (a t^{k_1}),e_{1,n+1} (b t^{k_2})] = 1.$$
\item For every $a, b \in \mathbb{F}_q$ and every $0 \leq k \leq n$,
$$e_{1,n+1} (a t^k) e_{1,n+1} (b t^k) = e_{1,n+1} ((a + b)t^k).$$
\end{enumerate}
\end{observation}

Let $r_1,...,r_n \in \mathbb{F}_q [t]$ of degree $\leq 1$. We denote $\Unip_{n+1} (r_1,...,r_n)$ to be the subgroup of $G$ generated by $\lbrace e_{i,i+1} (r_i) : 1 \leq i \leq n \rbrace$.
\begin{lemma}
\label{rel in subgrp lemma}
Let $G, K_{\lbrace 0 \rbrace},...,K_{\lbrace n \rbrace}, R_0,...,R_{n-1}$ as above. Fix $r_1,...,r_n \in \mathbb{F}_q [t]$ of degree $\leq 1$ $\Unip_{n+1} (r_1,...,r_n)$ and denote $r_i^j = \prod_{l=i}^{j -1} r_l$. Then every Steinberg relation of of of the following forms can be deduced from $\bigcup_i R_i$ and the number of relations needed is independent of $q$:
\begin{enumerate}
\item For every $1 \leq i < j \leq n+1$, and every $a, a \in \mathbb{F}_q$,
$$e_{i,j} (a_1 r_{i}^j) e_{i,j} (a_2 r_{i}^j) = e_{i,j} ((a_1 + a_2) r_{i}^j).$$
\item For every $1 \leq i_1 < j_1 \leq n+1, 1 \leq i_2 < j_2 \leq n+1$, and every $a_1, a_2 \in \mathbb{F}_q$,
$$[e_{i_1,j_1} (a_1 r_{i_1}^{j_1}),e_{i_2,j_2} (a_2 r_{i_2}^{j_2})] =
\begin{cases}
1 & i_1 \neq j_2, i_2 \neq j_1 \\
e_{i_1, j_2} (a_1 a_2 r_{i_1}^{j_2}) & j_1 = i_2 \\
e_{i_2, j_1} (- a_1 a_2 r_{i_2}^{j_1}) & j_2 = i_1
\end{cases}.$$
\end{enumerate}
\end{lemma}

\begin{proof}
We note that the subgroup of $G$ generated by $\lbrace e_{i,i+1} (r_i) : 1 \leq i \leq n \rbrace$ is isomorphic to $\Unip_{n+1} (\mathbb{F}_q)$ and thus we can apply Theorem \ref{warm-up bound on N_1 thm}.
\end{proof}

Using this Lemma we can prove the following Theorem:
\begin{theorem}
\label{bound on N_1 thm}
Let $G, K_{\lbrace 0 \rbrace},..., K_{\lbrace n-1 \rbrace}, R_0,...,R_{n-1}$ as above. If $n \geq 4$ or $q$ is odd, then
$$G = \langle \bigcup_i K_{\lbrace i \rbrace} \vert \bigcup_i R_i \rangle$$
and the Dehn function of this presentation is bounded independently of $q$.
\end{theorem}

\begin{proof}
Denote
$$e_{1,n+1} (at^k) = \begin{cases}
[e_{1,n} (t^k), e_{n,n+1} (a)] & k < n \\
[e_{1,n} (t^{k-1}), e_{n,n+1} (at)] & k = n \\
\end{cases}.$$
We only need to show that one can deduce each relation of types (1)-(4) of Observation \ref{rel obsrv 2} from a finite number of relations in $\bigcup_i R_i$ that is independent of $q$.

\textbf{Observation \ref{rel obsrv 2}, relations of type (1):} We need to show that $\bigcup_i R_i$ imply that for every $1 < j \leq n, 2 \leq i < n+1$, $i \neq j$, every $0 \leq k_1 \leq j-1, 0 \leq k_2 \leq n+1-i$ and every $a, b \in \mathbb{F}_q$
$$[e_{1,j} (a t^{k_1}),e_{i,n+1} (b t^{k_2})] = 1,$$
and that the number of relations needed to deduce this relation is independent of $q$. We will only prove this for $n=3$. The proof in the general case is similar, but more tedious and is left for the reader. In this case, either $j=2$ and $i=3$ or $j=3$ and $i=2$. If $j=2$ and $i=3$, then the needed relation appears in $R_1$ and we are done. Thus we are left to show that
$$[e_{1,3} (a t^{k_1}),e_{2,4} (b t^{k_2})] = 1,$$
where $a,b \in \mathbb{F}_q$ and $0 \leq k_1, k_2 \leq 2$. Applying Lemma \ref{rel in subgrp lemma} in all the cases where $r_1, r_2, r_3 \in \lbrace 1,t \rbrace$ prove the cases where $k_1 = k_2$ and the cases where $k_1 =1$ or $k_2 =1$. Thus, we are left to prove the cases
$$[e_{1,3} (a),e_{2,4} (b t^{2})] = 1,$$
and
$$[e_{1,3} (at^2),e_{2,4} (b)] = 1.$$
We will show only $[e_{1,3} (a),e_{2,4} (b t^{2})] = 1$: Applying Lemma \ref{rel in subgrp lemma} with $r_1 = t+a, r_2 = t+1, r_3 = bt$, we get that
$$[e_{1,3} (t^2 + (a+1)t +a),e_{2,4} (b t^{2} + bt)] = 1.$$
For the relations in $R_0 \cup R_2$, we have that
$$e_{1,3} (t^2 + (a+1)t +a) = e_{1,3} (t^2) e_{1,3} ((a+1)t) e_{1,3} (a),$$
$$e_{2,4} (b t^{2} + bt) = e_{2,4} (b t^2) e_{2,4} (bt).$$
We also already showed that
$$[ e_{1,3} (t^2), e_{2,4} (b t^2)] =1, [ e_{1,3} (t^2), e_{2,4} (b t)] =1,  [ e_{1,3} ((a+1)t), e_{2,4} (b t^2)] =1,$$
$$[ e_{1,3} ((a+1)t), e_{2,4} (b t)] =1, [ e_{1,3} (a), e_{2,4} (b t)] = 1,$$
and thus it follows from $[e_{1,3} (t^2 + (a+1)t +a),e_{2,4} (b t^{2} + bt)] = 1$, that
$$[e_{1,3} (a),e_{2,4} (b t^{2})] = 1,$$
as needed.

\textbf{Observation \ref{rel obsrv 2}, relations of type (2):} We need to show that $\bigcup_i R_i$ imply that for every $1 < j < n+1$, every $0 \leq k_1 \leq j-1, 0 \leq k_2 \leq n+1-j$ and every $a, b \in \mathbb{F}_q$,
$$[e_{1,j} (at^{k_1}),e_{j,n+1} (b t^{k_2})] = e_{1,n+1} (a b t^{k_1 + k_2}),$$
and that the number of relations needed to deduce this relation is independent of $q$. This follows from Lemma \ref{rel in subgrp lemma} by taking $r_1 =...= r_{k_1} =t, r_{k_1+1} =...=r_{j-1} =1, r_j = ... = r_{j+k_2-1}=t, r_{j+k_2} =...=r_n =1$.

\textbf{Observation \ref{rel obsrv 2}, relations of type (3):}  We need to show that $\bigcup_i R_i$ imply that for every $1 \leq i < j  \leq n+1$, every $0 \leq k_1 \leq j-i, 0 \leq k_2 \leq n$ and every $a, b \in \mathbb{F}_q$,
$$[e_{i,j} (a t^{k_1}),e_{1,n+1} (b t^{k_2})] = 1,$$
and that the number of relations needed to deduce this relation is independent of $q$. We note that since every elementary matrix in $G$ can be written as a product of elements of the form $e_{i,i+1} (a)$ and $e_{i,i+1} (a t)$ and the number of elements in such product is bounded independently of $q$. Thus, it is enough to show that for every $1 \leq i \leq n$,
$$[e_{i,i+1} (a),e_{1,n+1} (b t^{k_2})] = 1,$$
and
$$[e_{i,i+1} (at),e_{1,n+1} (b t^{k_2})] = 1.$$
If $k_2 <n$, then for every $i$, we can always choose $r_1,...,r_n \in \lbrace 1, t \rbrace$ such that exactly $k_2$ of them are $t$ and $r_i =1$. With this choice, applying Lemma \ref{rel in subgrp lemma}, implies that
$$[e_{i,i+1} (a),e_{1,n+1} (b t^{k_2})] = 1.$$
Similarly, if $k_2 >0$, we can apply Lemma \ref{rel in subgrp lemma} and show that
$$[e_{i,i+1} (at),e_{1,n+1} (b t^{k_2})] = 1.$$
Thus, we are left with the cases
$$[e_{i,i+1} (at),e_{1,n+1} (b)] = 1,$$
and
$$[e_{i,i+1} (a),e_{1,n+1} (bt^n)] = 1.$$
We will prove the first case. Apply Lemma \ref{rel in subgrp lemma} with
$r_i = at +b$ and all the other $r_j$'s equal $1$. We get that
$$[e_{i,i+1} (at +b),e_{1,n+1} (at +b)] = 1.$$
Note that
$$e_{i,i+1} (at +b) = e_{i,i+1} (at) e_{i,i+1} (b),$$
and that
\begin{dmath*}
e_{1,n+1} (at +b) = [e_{1,n} (1),e_{n,n+1} (at +b)] =
e_{1,n} (-1) e_{n,n+1} (-at-b)e_{1,n} (1) e_{n,n+1} (at+b)=
e_{1,n} (-1) e_{n,n+1}(-b) e_{n,n+1} (-at) e_{1,n} (1) e_{n,n+1} (at) e_{n,n+1} (b)=
e_{1,n} (-1) e_{n,n+1}(-b) e_{1,n} (1) [e_{1,n} (1), e_{n,n+1} (at)]  e_{n,n+1} (b)=
e_{1,n} (-1) e_{n,n+1}(-b) e_{1,n} (1)  e_{1,n+1} (at)  e_{n,n+1} (b)=
e_{1,n} (-1) e_{n,n+1}(-b) e_{1,n} (1)  e_{n,n+1} (b) e_{1,n+1} (at) =
e_{1,n+1} (b) e_{1,n+1} (at),
\end{dmath*}
where the last equality follows from the fact that we already proven that $e_{1,n+1} (at)$ commutes with all the elementary matrices. Thus, we have that
$$[e_{i,i+1} (at)e_{i,i+1} (b),e_{1,n+1} (at)e_{1,n+1} (b)] = 1.$$
We already showed that
$$[e_{i,i+1} (at),e_{1,n+1} (at)] =1, [e_{i,i+1} (b), e_{1,n+1} (at)] =1,$$
$$[e_{i,i+1} (b), e_{1,n+1} (b)] =1,$$
and thus it follows that
$$[e_{i,i+1} (at),e_{1,n+1} (b)] = 1,$$
as needed.

\textbf{Observation \ref{rel obsrv 2}, relations of type (4):} We need to show that $\bigcup_i R_i$ imply that for every $a, b \in \mathbb{F}_q$ and every $0 \leq k \leq n$,
$$e_{1,n+1} (a t^k) e_{1,n+1} (b t^k) = e_{1,n+1} ((a + b)t^k),$$
and that the number of relations needed to deduce this relation is independent of $q$. This follows from Lemma \ref{rel in subgrp lemma} with $r_1 = ... r_k = t$ and $r_{k+1} =... = r_n =1$.
\end{proof}

As a corollary, we get a generalization of Theorem \ref{new coboundary expanders thm intro} that appeared in the introduction:
\begin{corollary}
\label{N_0 + N_1 bound for unip coro}
Let $G, K_{\lbrace 0 \rbrace},...,K_{\lbrace n-1 \rbrace}$ as above and let $X = X(G, (K_{\lbrace i \rbrace})_{i \in \lbrace 0,...,n-1 \rbrace})$ be the coset complex. If $n \geq 4$ or $q$ is odd then the constants $N_0$ and $N_1$ of $X$ are bounded independently of $q$ and thus $\Exp_b^0 (X), \Exp_b^1 (X)$ are bounded from below independently of $q$ (the bound does depend on $n$).
\end{corollary}

\begin{proof}
By \cite[Theorem 3.5]{KOCosetGeom} and Theorem \ref{transitive action thm} stated above, $X = X(G, (K_{\lbrace i \rbrace})_{i \in \lbrace 0,...,n-1 \rbrace})$ is strongly symmetric. Thus, combining Corollary \ref{bound on N_0 coro}, Theorem \ref{bound on N_1 thm} and Theorem \ref{N_0 + N_1 bound thm} yields the desired result.
\end{proof}

\section{New cosystolic and topological expanders}
\label{New cosys expanders sec}

After all this, we are ready to prove Theorem \ref{new cosystolic expanders thm intro} from the introduction. Let us state it again for completeness:

\begin{theorem}
\label{new cosystolic expanders thm}
Let $s \in \mathbb{N}, s > 4$ and $q$ be a prime power. Denote $G^{(s)}_q$ to be the group of $4 \times 4$ matrices with entries in $\mathbb{F}_q [t] / \langle t^s \rangle$ generated by the set
$$\lbrace e_{1,2} (a +bt), e_{2,3} (a +bt), e_{3,4} (a +bt), e_{4,1} (a +bt)  : a,b \in \mathbb{F}_q \rbrace.$$
For $0 \leq i \leq 2$, define $K_{\lbrace i \rbrace}$ to be the subgroup of $G^{(s)}_q$ generated by
$$\lbrace e_{j,j+1} (a +bt),  e_{4,1} (a +bt) : a,b \in \mathbb{F}_q, 1 \leq j \leq 3, j \neq i+1 \rbrace$$
and define $K_{\lbrace 3 \rbrace}$ to be the subgroup of $G^{(s)}_q$ generated by
$$\lbrace e_{1,2} (a +bt), e_{2,3} (a +bt), e_{3,4} (a +bt)  : a,b \in \mathbb{F}_q \rbrace.$$
Denote $X^{(s)}_q = X(G^{(s)}_q, (K_{\lbrace i \rbrace})_{i \in \lbrace 0,...,3 \rbrace})$ to be the coset complex as defined above. Then for any fixed $q$, $\lbrace X^{(s)}_q \rbrace_{s \geq 5}$ is a family of bounded degree simplicial complexes and if $q$ is odd and large enough, then there are $\varepsilon >0, \mu >0$ such that for every $s$, the $2$-skeleton of $X^{(s)}$ is a a $(\varepsilon , \mu)$-cosystolic expander. Thus, the sequence of $2$-skeleton of $X^{(s)}_q$ is a sequence of bounded degree cosystolic and topological expanders.
\end{theorem}

\begin{proof}
We start by proving that there is $\varepsilon ' >0$ such that for every $s >4$ and every odd $q$, all the links of $X^{(s)}_q$ are $\varepsilon '$-coboundary expanders. Since $X^{(s)}_q$ is $3$-dimensional, we have to consider links of edges and links of vertices.

For links of edges, it is shown in \cite{KOconstruction} that every link is a bipartite graph with a second eigenvalue $\leq \frac{1}{\sqrt{q}}$. In this case, $\Exp_b^0$ is equal to the Cheeger of the graph and thus by the Cheeger inequality it is always larger than $\frac{1}{2}-\frac{1}{2 \sqrt{q}} \geq \frac{1}{2}-\frac{1}{2 \sqrt{3}}$.

For links of vertices, we note that every link is exactly the simplicial complex discussed in Theorem \ref{new coboundary expanders thm intro} and thus these links are coboundary expanders with expansion that does not depend on $q$.

Let $\lambda = \lambda (\varepsilon ')$, $\varepsilon = \varepsilon (\varepsilon ', \lambda) >0$ and $\mu = \mu (\varepsilon ', \lambda) >0$ be the constants of Theorem \ref{EK criterion thm}. For every odd $q$ such that $\frac{1}{\sqrt{q}-3} \leq \lambda$, we have that $X^{(s)}$ is a $\lambda$-local spectral expander and thus by Theorem \ref{EK criterion thm} the $2$-skeleton of $X^{(s)}$ is a $(\varepsilon, \mu)$-cosystolic expander.
\end{proof}

The motivation behind the definition of cosystolic expansion was to prove topological overlapping: Let $X$ be an $n$-dimensional simplicial complex as before. Given a map $f : X^{(0)} \rightarrow \mathbb{R}^n$, a topological extension of $f$ is a continuous map $\widetilde{f} : X \rightarrow \mathbb{R}^n$ which coincides with $f$ on $X^{(0)}$.
\begin{definition}[Topological overlapping]
\label{top exp def}
A simplicial complex $X$ as above is said to have $c$-topological overlapping (with $1 \geq c>0$) if for every $f: X^{(0)} \rightarrow \mathbb{R}^n$ and every topological extension $\widetilde{f}$, there is a point $z \in \mathbb{R}^n$ such that
$$\vert \lbrace \sigma \in X (n) : z \in \widetilde{f} (\sigma)\rbrace \vert \geq c \vert X (n) \vert .$$
In other words, this means that at least $c$ fraction of the images of $n$-simplices intersect at a single point.

A family of pure $n$-dimensional simplicial complexes $\lbrace X_j \rbrace$ is called a family of topological expanders, if there is some $c>0$ such that for every $j$, $X_j$ has $c$-topological overlapping.
\end{definition}

In \cite{DKW}, it was shown that cosystolic expansion implies topological expansion and thus as a Corollary of \cite[Theorem 8]{DKW} and Theorem \ref{new cosystolic expanders thm} we get that:
\begin{theorem}
\label{new top exp thm}
There is a constant $Q$ such that for every odd prime power $q \geq Q$, such there is $c >0$ such that for every $s >4$, the $2$-skeleton of $X^{(s)}_q$ is $c$-topological overlapping, i.e., the family $\lbrace 2 \text{-skeleton of } X^{(s)}_q \rbrace_{s >4}$ is a family of bounded degree topological expanders.
\end{theorem}

\appendix

\section{The existence of a cone function and vanishing of (co)homology}
\label{The existence of a cone function and vanishing of (co)homology sec}

As seen in Examples \ref{0-Cone example}, \ref{1-Cone example} above, for a simplicial complex $X$ and $k \geq 0$, a $k$-cone function may not exist and if it exists it may not be unique. The existence of a $k$-cone function turns out to be equivalent to vanishing of (co)homology (we recall that by the universal coefficient theorem the vanishing of the $j$-th homology with coefficients in $\mathbb{F}_2$ is equivalent to the vanishing of the $j$-th cohomology with coefficients in $\mathbb{F}_2$):
\begin{proposition}
\label{existence of cone func prop}
Let $X$ be a finite $n$-dimensional simplicial complex and $0 \leq k \leq n-1$. There exists a $k$-cone function (with some apex) if and only if $\widetilde{H}_j (X)=\widetilde{H}^j (X) = 0$ for every $0 \leq j \leq k$.
\end{proposition}

\begin{proof}[Proof of Proposition \ref{existence of cone func prop}]
Let $X$ be a finite $n$-dimensional simplicial complex and $0 \leq k \leq n-1$.

Assume first that for every $0 \leq j \leq k$, $\widetilde{H_j} (X) = 0$. Then by definition, for every $0 \leq j \leq k$, $\Sys_j (X) = \infty$. Thus conditions of Proposition \ref{filling constants prop} are fulfilled trivially and as a result there exists a $k$-cone function.

In the other direction, assume there exists a $k$-cone function $\Cone_k^v$. Let $A \in Z_k (X)$. Then by the definition of the cone function,
$$\partial_{k+1} \Cone_k^v (A) = A + \Cone_k^v (\partial_{k} A) = A + \Cone_k^v (0) = A,$$
thus $A \in B_k (X)$ and since this holds for every $A$, it follows that $\widetilde{H}^k (X) = 0$.

By definition, the existence of a $k$-cone function implies the existence of a $j$-cone function for every $0 \leq j \leq k$ and therefore the above argument shows that $\widetilde{H}^j (X) = 0$ for every $0 \leq j \leq k$.
\end{proof}

\bibliographystyle{alpha}
\bibliography{bibl1}
\end{document}